\documentclass{neutral}

\title{Numerical resolution by the quasi-reversibility method of a data completion problem for Maxwell's equations}

\addauthor{Marion Darbas}{marion.darbas@u-picardie.fr}{LAMFA CNRS UMR 7352 - Université de Picardie Jules Verne, 33 rue Saint-Leu, 80039 Amiens cedex 1}
\addauthor{Jérémy Heleine}{jeremy.heleine@u-picardie.fr}{LAMFA CNRS UMR 7352 - Université de Picardie Jules Verne, 33 rue Saint-Leu, 80039 Amiens cedex 1}
\addauthor{Stephanie Lohrengel}{stephanie.lohrengel@univ-reims.fr}{LMR CNRS UMR 9008 - Université de Reims Champagne-Ardenne, Moulin de la Housse, 51687 Reims cedex 2}


\drafted{February 14, 2020}
\submitted{February 14, 2020}
\accepted{July 17, 2020}
\pubjournal{Inverse Problems and Imaging}

\newcommand{\bfH}{\mathbf{H}}

\begin{document}

	\maketitle

	\begin{abstract}
		This paper concerns the numerical resolution of a data completion problem for the time-harmonic Maxwell equations in the electric field.
		The aim is to recover the missing data on the inaccessible part of the boundary of a bounded domain from measured data on the accessible part.
		The non-iterative quasi-reversibility method is studied and different mixed variational formulations are proposed. Well-posedness, convergence and regularity results are proved.
		Discretization is performed by means of edge finite elements. Various two- and three-dimensional numerical simulations attest the efficiency of the method, in particular for noisy data.
	\end{abstract}

	\displaykeywords

	\section{Introduction}

	Let $\Omega$ denote a bounded and simply connected domain in $\R^3$ of boundary $\Gamma \coloneqq \partial\Omega$. The unit outward normal to $\Omega$ is denoted by $\bfn$. Assume the medium in $\Omega$ to be inhomogeneous and isotropic, of constant magnetic permeability $\mu = \mu_0$ with $\mu_0$ the magnetic permeability in vacuum. Let $\eps, \sigma$ be non-negative functions representing the electric permittivity and conductivity respectively. The refractive index $\kappa$ of the medium in $\Omega$ is defined by $\kappa(\bfx) = \frac{1}{\eps_0}\left(\eps(\bfx) + i\frac{\sigma(\bfx)}{\omega}\right), \bfx \in \Omega$, where $\eps_0$ is the electric permittivity in vacuum. We consider the Cauchy problem for the time-harmonic Maxwell equations at a frequency $\omega > 0$
	\begin{equation}
		\label{eq:maxwell}
		\curl\curl\bfE - k^2\kappa\bfE = 0, \quad \stext[r]{in} \Omega,
	\end{equation}
	with $\bfE$ the electric field intensity and $k \coloneqq \omega\sqrt{\mu_0\eps_0}$ the wavenumber. Namely, we assume that the Cauchy data $\bfE\times\bfn$ and $\curl\bfE\times\bfn$ are known on a part $\Gamma_0$
	of the boundary $\Gamma$, called the accessible part. The aim is to solve the data completion problem which consists in recovering the missing data on the inaccessible part $\Gamma\setminus\Gamma_0$ from overdetermined
	data on the accessible part. This covers for instance medical applications where measurements are taken at electrodes placed at specific positions on the skin (of the head or the chest). The numerical resolution of this class of inverse problems
	is challenging in many application fields such as non-destructive testing or medical imaging (electroencephalography, brain stroke detection) since only partial noisy data are available. Inversion methods are in general
	more efficient for total data known on the entire boundary. Data completion algorithms provide interesting preliminary tools to get an input for the numerical resolution of inverse problems (e.g~\cite{BourgeoisDardeIPI}).

	Cauchy problems are well-known to be ill-posed (e.g.~\cite{Alessandrini09, BenBelgacem}). Various regularization methods have been introduced for reconstructing the missing data for elliptic equations.
	There is an extensive literature on the numerical resolution of the problem. Most of the methods are iterative and are based on solving one or more well-posed direct problems at each iteration. They can thus be  numerically expensive.
	Classical least square methods with a Tikhonov regularization technique are commonly used~(see e.g. \cite{Cimetiere}).
	There are several techniques for the optimal choice of the regularization parameter such as the L-curve method and the Morozov discrepancy principle.
	Without being exhaustive, we can cite other approaches: a reconstruction of missing data based on the minimization of an energy functional is adopted in~\cite{Andrieux}. In~\cite{Ling}, a boundary
	control technique is used to obtain an approximation of the missing Dirichlet boundary data for the Laplace equation.
	An alternating iterative method to solve data completion for the Laplace equation and the Lamé system
	is originally introduced in~\cite{Kozlov} and is widely used and studied in the literature. In this vein, an approach combining the Steklov-Poincaré variational formulation of elliptic second order equations and the Lavrentiev's method has been proposed
	and validated in~\cite{BenBelgacem05, Azaiez, Azaiez11}.
	Methods based on integral equations have been also considered \cite{CakoniKress,Chapko,Chapko2}.

	Very few papers are concerned with the data completion problem for the vector Maxwell equations. We cite \cite{MeliSlo06,SloMeli10} where the Cauchy problem for a steady state eddy-current model problem is studied and an iterative algorithm based on a least-square functional is proposed.

	In this paper, we focus on the quasi-reversibility method which is a non-iterative approach. It was introduced by Lattès and Lions in~\cite{LattesLions67}. The idea is to replace the ill-posed Cauchy problem by a well-posed variational problem with additional unknowns which depends on a small regularization parameter. The variational setting is numerically interesting since finite element methods can be used.  The quasi-reversibility method has been successfully adopted and validated for Laplace's equation (see e.g.~\cite{Klibanov91,Bourgeois05,BourgeoisDarde,Darde}) and Helmholtz's equation~\cite{BR18}. We propose to study the quasi-reversibility method for the vector Maxwell equations. To the best of our knowledge, it's the first time that such an approach is studied for solving a data completion problem in electromagnetics. It is worth noticing that the functional framework of the vector Maxwell equations is substantially different from the Laplace equation or even the Lamé system. The natural vector space is given by fields with square integrable curl and a special attention has to be paid to the definition of the traces at the boundary. Furthermore, the vector Maxwell problem is numerically challenging due to the computational complexity, especially in three dimensions and the large kernel of the discrete $\curl$-$\curl$-operator. The conditioning of the corresponding matrix deteriorates with the mesh refinement and preconditioners have to be prescribed when iterative solvers are used (e.g.~\cite{GreifSchotzau, HiptmairXu}).

	The article is organized as follows. \autoref{sec:problem} is devoted to the setting of the problem. \autoref{sec:QR} presents a first version of the quasi-reversibility method for solving the data completion problem for Maxwell's equations in the electric field. Both theoretical and numerical aspects are addressed. \autoref{sec:RRQR} proposes two relaxed mixed problems with regularization in order to deal with noisy data. The relaxed problem is shown to be well posed, and convergence to the solution of the initial Cauchy problem is proven under some conditions on the involved regularization and relaxation parameters. A regularity result is given for one of the formulations. Numerical simulations confirm the theoretical results in two and three dimensions of space and attest the efficiency of the method. The choice of the different parameters is discussed. \autoref{sec:tikhonov} is a remark about the solutions of the proposed formulations being seen as the critical points of real-valued cost functions. Finally, we give some concluding remarks.

	\section{Position of the problem. Notations.}
	\label{sec:problem}

	In the sequel, define $\Gamma_0$, the accessible part of $\Gamma$ \cite{Brownetal16}, according to the following
	\begin{definition}
		\label{Def_accessible}
		Let $\Omega$ be a non-empty, open, bounded connected domain in $\R^3$ with Lipschitz boundary $\Gamma$. Let $\Gamma_0$ a smooth non-empty open subset of $\Gamma$ such that meas $\Gamma_0 > 0$. The part $\Gamma_0$ is called the accessible part of the boundary $\Gamma$ and $\Gamma_1 \coloneqq \Gamma \setminus \overline{\Gamma_0}$ the inaccessible part.
	\end{definition}

	\begin{remark}
		In this paper, we address the theoretical part in the three-dimensional case. All the results hold true in two dimensions as well, noticing that two curl-operators have to be distinguished in this case. Indeed, the curl of vector fields $\bfE = (E_1,E_2)$ with two components is defined by
		\[ \operatorname{curl}\bfE = \partial_1 E_2 - \partial_2 E_1, \]
		whereas the vector curl of a scalar function $\varphi$ is given by
		\[ \curl\varphi = \begin{pmatrix}
		\partial_2\varphi\\ -\partial_1\varphi
		\end{pmatrix}.
		\]
		Consequently, the operator of the partial differential equation reads $\curl\operatorname{curl}$. Numerical results are given for both 2D and 3D configurations.
	\end{remark}

	Let us make precise the functional setting of the data completion problem. The fundamental vector space for Maxwell's equation is given by
	$\Hcurl[\Omega] = \set*{\bfu \in L^2(\Omega)^3}{\curl\bfu \in L^2(\Omega)^3}$. For any vector field $\bfu \in \Hcurl[\Omega]$, we define the tangential trace by continuous extension of the mapping $\gamma_t(\bfu) \coloneqq \restriction{\bfu}{\Gamma} \times \bfn$. We introduce the trace space
	\begin{equation}\label{eq:YG}
		Y(\Gamma) = \set*{\bff \in H^{-1/2}(\Gamma)^3}{\Exists{\bfu \in \Hcurl[\Omega]} \gamma_t(\bfu) = \bff}
	\end{equation}
	and its restriction to $\Gamma_0$ in the distributional sense $Y(\Gamma_0) = \set*{\restriction{\bff}{\Gamma_0}}{\bff \in Y(\Gamma)}$. The following lemma holds.

	\begin{lemma}
		\label{lm:trace}
		The partial trace application $\bfu \mapsto \restriction{\gamma_t(\bfu)}{\Gamma_0}$ is linear, continuous and surjective from $\Hcurl[\Omega]$ to $Y(\Gamma_0)$. Moreover, there exists a continuous lifting application: we can find a constant $r > 0$ such that, for any $\bfv \in Y(\Gamma_0)$, there exists $\bfu \in \Hcurl[\Omega]$ with $\restriction{\gamma_t(\bfu)}{\Gamma_0} = \bfv$ and $\norm{\bfu}{\Hcurl[\Omega]} \leq r \norm{\bfv}{Y(\Gamma_0)}$.
	\end{lemma}

	\autoref{lm:trace} is a direct consequence of the fact that the trace application $\gamma_t$ is linear, continuous and surjective from $\Hcurl[\Omega]$ to $Y(\Gamma)$ (see \cite{Monk03}). Then, we build an homeomorphism between $\Hcurl[\Omega] / V_0$ (see~\eqref{defVf}) and $Y(\Gamma_0)$.

	For further purposes, we also introduce the second trace operator $\gamma_T$ that is defined by $\gamma_T(\bfu) = (\bfn\times\bfu_{|\Gamma})\times\bfn$ for smooth vector fields and is well defined as an operator
	from $\Hcurl[\Omega]$ on the dual space of $Y(\Gamma)$, $Y(\Gamma)^\prime$
	(see \cite{Monk03} for details). The trace spaces for fields in $\Hcurl[\Omega]$ can be characterized more precisely depending on the regularity of the boundary (see \cite{Monk03, Buffa01, Buffa01bis}). However, for our study, the above abstract definition is sufficient.

	The set of admissible coefficients $\eps$ and $\sigma$ is given in the following definition.

	\begin{definition}
		\label{Def_admissible}
		The pair of coefficients $\eps$ and $\sigma$ is admissible if $\eps, \sigma \in \mcC^1(\overline{\Omega})$ such that $\eps(\bfx) \geq \tilde{\eps}$ and  $\sigma(\bfx) \geq 0$ almost everywhere in $\Omega$ for a strictly positive constant $\tilde{\eps}$.
	\end{definition}

	\begin{definition}
		\label{Def_Cauchyset}
		Let $\Omega$ and $\Gamma_{0}$ be as in \autoref{Def_accessible}. For a pair of admissible coefficients $(\eps, \sigma)$ defined on $\Omega$ as in \autoref{Def_admissible}, the corresponding Cauchy data set $C(\eps,\sigma; \Gamma_0)$ at a fixed frequency $\omega>0$ consists of pairs $(\bff, \bfg) \in Y(\Gamma_0) \times Y(\Gamma_0)$ such that there exists a field $\bfE \in \Hcurl[\Omega]$ satisfying
		\begin{equation}
			\label{eq:Cauchy}
			\left\{
			\begin{array}{rcl@{\hspace{4\tabcolsep}}l}
				\curl\curl\bfE - k^2\kappa\bfE &=& 0, & \stext[r]{in} \Omega, \\
				\gamma_t(\bfE) &=& \bff, & \stext[r]{on} \Gamma_0, \\
				\gamma_t(\curl\bfE) &=& \bfg, & \stext[r]{on} \Gamma_0,
			\end{array}
			\right.
		\end{equation}
		where $\kappa = (\eps + i \sigma/\omega)/\eps_0$.
	\end{definition}

	The definition of Cauchy data sets is used for instance in~\cite{Ola03, Caro11,CaroZhou14}. The partial boundary data are also given by the admittance map $\Lambda\colon \restriction{(\bfE\times\bfn)}{\Gamma} \mapsto \restriction{(\curl\bfE\times\bfn)}{\Gamma}$ restricted to $\Gamma_0$ if $\omega$ is not a resonant frequency for \eqref{eq:maxwell}.

	The Cauchy problem \eqref{eq:Cauchy} is known to be ill-posed in the sense that it may not possess a solution for arbitrary data $(\bff, \bfg)$ in $Y(\Gamma_0)\times Y(\Gamma_0)$ and that the solution is not stable with respect to the data. On the other hand, if the Cauchy problem \eqref{eq:Cauchy}
	admits a solution, this one is unique.
	This is an immediate consequence of the following Holmgren-type theorem (cf. \cite{Brownetal16}, Lemma 5.4):
	\begin{lemma}
		\label{lem:holmgren}
		Let $\Omega$ be a bounded connected Lipschitz domain. Let $\kappa\in \mathcal{C}^1(\overline{\Omega})$ such that $\real{\kappa} \ge c$ for a given constant $c>0$. Let $\Gamma_0$ be an open non-empty subset of $\partial\Omega$. If $\bfu\in \Hcurl[\Omega]$ is such that $\curl\curl\bfu\in L^2(\Omega)^3$ and
		\[
			\left\{
			\begin{array}{rcl@{\hspace{4\tabcolsep}}l}
				\curl\curl\bfu - k^2\kappa\bfu &=& 0, & \stext[r]{in} \Omega, \\
				\gamma_t(\bfu) &=& 0, & \stext[r]{on} \Gamma_0, \\
				\gamma_t(\curl\bfu) &=& 0, & \stext[r]{on} \Gamma_0,
			\end{array}
			\right.
		\]
		then $\bfu\equiv 0$ on $\Omega$.
	\end{lemma}

	\begin{remark}
		The question whether a solution of the Cauchy problem exists for given data or not is difficult from a theoretical point of view (see \cite{Hadamard,BenBelgacem}). It is not addressed in the present paper. From a practical point of view, however, one can reasonably assume that the given data correspond to (noisy) measurements and given sources. These data are the consequence of an electric field that naturally satisfies the Cauchy problem. In the sequel, we therefore assume that the data belong to the Cauchy set associated with the Cauchy problem.
	\end{remark}

	\section{The quasi-reversibility method for Maxwell's equations}
	\label{sec:QR}

	From \autoref{lem:holmgren} and the \autoref{Def_Cauchyset} of Cauchy data sets, the Cauchy problem \eqref{eq:Cauchy} admits a unique solution for any couple $(\bff,\bfg)$ belonging to $\mathcal{C}(\eps,\sigma;\Gamma_0)$. Assuming that the set $\mathcal{C}(\eps,\sigma;\Gamma_0)$ is not empty, the rest of the paper is devoted to the numerical resolution of this data completion problem for Maxwell's equations.

	\subsection{Principle for data belonging to the trace space}

	The method of quasi-reversibility (called hereafter QR method) provides a regularized solution of Cauchy problems in a bounded domain. An abstract setting of the method is presented in~\cite{BR18}. Here, we prove that the abstract setting fits with the functional framework of Maxwell's equations.

	Let us introduce the spaces
	\begin{equation}
		\label{defVf}
		V_\bff = \set*{\bfu \in \Hcurl[\Omega]}{\gamma_t(\bfu) = \bff \stext{on} \Gamma_0}
	\end{equation}
	which is well defined for any $\bff \in Y(\Gamma_0)$, and
	\[M = \set*{\bfmu \in H(\curl)}{\gamma_t(\bfmu) = 0 \stext{on} \Gamma_1}.\]

	One may notice that a field $\bfmu$ belongs to the vector space $M$ if and only if
	\[ (\curl\bfmu,\Phi) = (\bfmu,\curl\Phi)\ \forall \Phi \in \  \stackrel{\circ\;}{\mathcal{C}^\infty_{\Gamma_0}}(\Omega)^3\]
	where the space $\stackrel{\circ\;}{\mathcal{C}^\infty_{\Gamma_0}}(\Omega)$
	contains the restriction to $\Omega$ of smooth functions $\phi$ with compact support in $\R^3$ such that ${\rm dist}({\rm supp}(\phi), \Gamma_0)>0$, in other words $\phi$ vanishes in a neighborhood of $\Gamma_0$.

	Assume that $(\bff,\bfg) \in C(\eps,\sigma; \Gamma_0)$ (see \autoref{Def_Cauchyset}) where $\eps$ and $\sigma$ are admissible coefficients (see \autoref{Def_admissible}). We denote by $a(\cdot,\cdot)$ the sesqui-linear form corresponding to \eqref{eq:maxwell}:
	\begin{equation}
		\label{def_a}
		a(\bfu,\bfv) = (\curl\bfu,\curl\bfv) - k^2(\kappa\bfu,\bfv)
	\end{equation}
	with $\bfu, \bfv \in \Hcurl[\Omega]$ and $(\cdot,\cdot)$ the dot-product in $L^2$.
	On $M$, we introduce the linear form $\ell(\cdot)$ defined by
	\begin{equation}
		\label{def_l}
		\ell(\bfpsi) = \duality{\bfg}{\gamma_T(\bfpsi)}{Y(\Gamma_0),Y(\Gamma_0)'}.
	\end{equation}

	\begin{proposition}
		\label{p:weakCauchy}
		Let $\bfE \in \Hcurl[\Omega]$ be a solution of the Cauchy problem \eqref{eq:Cauchy} for a couple $(\bff,\bfg)\in\mathcal{C}(\eps,\sigma;\Gamma_0)$. Then $\bfE$ belongs to $V_\bff$ and satisfies the following variational equation
		\begin{equation}
			\label{eq:weakCauchy}
			a(\bfE,\bfpsi) = \ell(\bfpsi)\ \forall\bfpsi\in M.
		\end{equation}
		Reciprocally, any solution of \eqref{eq:weakCauchy} satisfies \eqref{eq:Cauchy} in the distributional sense.
	\end{proposition}

	\begin{proof}
		According to the definition of the Cauchy data set, $\bff$ and $\bfg$ belong to the trace space $Y(\Gamma_0)$. We deduce from a density result in \cite{BPS16} that for a given $\bfpsi \in M$, there is a sequence of fields $\bfpsi_n \in \ \stackrel{\circ\text{\;}}{\mathcal{C}^\infty_{\Gamma_1}}(\Omega)$ such that $\bfpsi_n\longrightarrow \bfpsi$ in
		$\Hcurl[\Omega]$. Since $\bfpsi_n$ vanishes identically on a neighborhood of $\Gamma_1$, its restriction to $\Gamma_0$ belongs to $\mathcal{D}(\Gamma_0)$. Any solution $\bfE$ of the Cauchy problem thus satisfies
		\[\duality{\gamma_t(\curl\bfE)}{\restriction{\bfpsi_n}{\Gamma_0}}{} = \duality{\bfg}{\restriction{\bfpsi_n}{\Gamma_0}}{} = \duality{\bfg}{\gamma_T(\bfpsi_n)}{\Gamma_0}\]
		since $\bfg\cdot\bfn = 0$ by definition of the trace space $Y(\Gamma)$.
		It follows from the classical Green's theorem in $\Hcurl[\Omega]$ that
		\begin{align*}
			\dotprod{\curl\curl\bfE}{\bfpsi_n}{} &= \dotprod{\curl\bfE}{\curl\bfpsi_n}{} - \duality{\gamma_t(\curl\bfE)}{\restriction{\bfpsi_n}{\Gamma_0}}{} \\
			&= \dotprod{\curl\bfE}{\curl\bfpsi_n}{} - \duality{\bfg}{\gamma_T(\bfpsi_n)}{\Gamma_0}.
		\end{align*}
		Taking the limit $n \to \infty$ then yields
		\[\dotprod{\curl\curl\bfE}{\bfpsi}{} = \dotprod{\curl\bfE}{\curl\bfpsi}{} - \duality{\bfg}{\gamma_T(\bfpsi)}{\Gamma_0}\]
		for any $\psi \in M$ since the trace operator $\gamma_T$ is continuous on $\Hcurl[\Omega]$.
		Inversely, if $\bfE \in V_\bff$ is a variational solution of \eqref{eq:weakCauchy}, taking $\psi \in \mathcal{C}^\infty_0(\Omega)^3$ yields immediately the partial differential equation of \eqref{eq:Cauchy}. The first boundary condition is included in the variational space $V_\bff$. Now, let $\varphi \in \mathcal{C}^\infty_0(\Gamma_0)^3$ be a regular vector field defined on the boundary with support in $\Gamma_0$ and such that $\varphi \cdot \bfn = 0$. $\varphi$ is the restriction to $\Gamma_0$ of a vector field $\bfpsi \in\ \stackrel{\circ\text{\;}}{ \mathcal{C}^\infty_{\Gamma_1}}(\Omega)^3$:  $\restriction{\bfpsi}{\Gamma_0} = \varphi$. Obviously, $\bfpsi$ belongs to the vector space $M$. The variational equation \eqref{eq:weakCauchy} and Green's formula yield
		\begin{align*}
			\duality{\bfg}{\varphi}{\Gamma_0} &= \dotprod{\curl\bfE}{\curl\bfpsi}{} - k^2\dotprod{\kappa\bfE}{\bfpsi}{} \\
			&= \duality{\gamma_t(\curl\bfE)}{\varphi}{\Gamma_0}
		\end{align*}
		and the second boundary condition follows.
	\end{proof}

	For small $\delta > 0$, we now consider the following weak mixed formulation: Find $(\bfE_\delta,\bfF_\delta) \in  V_\bff \times M$ such that
	\begin{equation}
		\label{eq:qr}
		\left\{
		\begin{array}{rcl@{\hspace{4\tabcolsep}}l}
			\delta\dotprod{\bfE_\delta}{\bfphi}{\Hcurl[\Omega]} + a(\bfphi,\bfF_\delta) &=& 0, & \forall \bfphi \in V_0, \\
			a(\bfE_\delta,\bfpsi) - \dotprod{\bfF_\delta}{\bfpsi}{\Hcurl[\Omega]} &=& \ell(\bfpsi), & \forall \bfpsi \in M,
		\end{array}
		\right.
	\end{equation}
	where $\dotprod{}{}{\Hcurl[\Omega]}$ denotes the dot-product in $\Hcurl[\Omega]$. \autoref{thm:convergence} states that the regularized solution $(\bfE_\delta, \bfF_\delta)$ tends to $(\bfE, 0)$
	when $\delta$ tends to 0, where $\bfE$ is solution to~\eqref{eq:Cauchy}.

	\begin{theorem}
		\label{thm:convergence}
		Let $(\bff,\bfg) \in Y(\Gamma_0) \times Y(\Gamma_0)$. For any $\delta > 0$, problem~\eqref{eq:qr} admits a unique solution $(\bfE_\delta,\bfF_\delta) \in  V_\bff \times M$.

		If, in addition, $(\bff,\bfg)$ belongs to the Cauchy data set $C(\eps, \sigma; \Gamma_0)$, then
		\begin{equation}
			\label{eq:convergence}
			\lim_{\delta \to 0} (\bfE_\delta,\bfF_\delta) = (\bfE, 0)
		\end{equation}
		in $V_f\times M$.
		Here, $\bfE\in V_{\bff}$ is the unique solution of the Cauchy problem \eqref{eq:Cauchy}.
	\end{theorem}

	\begin{proof}
		Since the trace operator $\gamma_t$ is onto according to \autoref{lm:trace}, we can apply Theorem 2.4 from \cite{BR18}. Problem~\eqref{eq:qr} can thus be written in a closed form as a classical variational formulation in the unknowns $(\bfE_\delta,\bfF_\delta)$, involving a continuous and coercive sesqui-linear form $A_\delta\left((\cdot, \cdot)\right)$ defined on $\Hcurl[\Omega] \times M$. Existence and uniqueness follow from Lax-Milgram's Lemma.
		Assume in addition that the couple $(\bff,\bfg)$ belongs to the Cauchy set $\mathcal{C}(\eps,\sigma;\Gamma_0)$. Hence, there exists a field $\bfE\in V_\bff$, solution to the Cauchy problem {\eqref{eq:Cauchy}}. From \autoref{lem:holmgren}, we infer that $\bfE$ is unique. The convergence of the sequence $(\bfE_\delta,\bfF_\delta)$ to $(\bfE,0)$ now follows as in \cite{BR18} since $\bfE$ is the only solution to the variational formulation.
	\end{proof}

	\subsection{Discretization by FEM and numerical results}
	\label{sec:numerical_classic}

	A first series of numerical tests for two dimensional results investigates in which way the size of the accessible part $\Gamma_0$ influences the quality of the data completion procedure. In view of possible applications of the QR method in inverse coefficient problems, we are particularly interested in ring-like computational domains $\Omega$. Indeed, in many configurations, the material parameters are known in a ring-like neighborhood of the boundary, and data completion has to be applied to map the available information from part of the outer boundary to the inner boundary.

	The numerical solver for the Maxwell equations has been implemented with FreeFem++ (see~\cite{Hecht12}). We consider a regular discretization $\mcT_h$ of $\overline{\Omega}$. For any element $T \in \mcT_h$, let $h_T$ be its diameter. Then $h = \max_{T \in \mcT_h} h_T$ is the mesh parameter of $\mcT_h$. Conforming edge finite elements of order 1 (see~\cite{Nedelec80}) are used to approximate the solution of the regularized problem~\eqref{eq:qr}.
	The resulting discrete formulation reads as follows: Find $(\bfE_\delta^h,\bfF_\delta^h)\in V_\bff^h\times M^h$ such that
	\begin{equation}
		\label{eq:qr-discret}
		\left\{
		\begin{array}{rcl@{\hspace{4\tabcolsep}}l}
			\delta\dotprod{\bfE_\delta^h}{\bfphi^h}{\Hcurl[\Omega]} + a(\bfphi^h,\bfF_\delta^h) &=& 0, & \forall \bfphi^h \in V_0^h, \\
			a(\bfE_\delta^h,\bfpsi^h) - \dotprod{\bfF_\delta^h}{\bfpsi^h}{\Hcurl[\Omega]} &=& \ell(\bfpsi^h), & \forall \bfpsi^h \in M^h.
		\end{array}
		\right.
	\end{equation}

	\begin{proposition}
		\label{p:discretization}
		For any $\delta>0$, problem \eqref{eq:qr-discret} admits a unique solution. If in addition, the solution fields $\bfE_\delta$ and $\bfF_\delta$ of problem \eqref{eq:qr} belong to $H^s(\curl;\Omega)=\set*{\bfu \in H^s(\Omega)^3}{\curl\bfu\in H^s(\Omega)^3}$ with $1/2< s \le 1$,  the following error estimate holds true
		\begin{equation}\label{eq:discretization}
			\norm{(\bfE_\delta - \bfE_\delta^h,\bfF_\delta-\bfF_\delta^h)}{\Hcurl[\Omega]}^2 \le C \frac{h^s}{\min(\delta,1)} \norm{(\bfE_\delta,\bfF_\delta)}{H^s(\curl;\Omega)}^2.
		\end{equation}
	\end{proposition}

	\begin{proof}
		Since problem~\eqref{eq:qr} can be written as a variational formulation with a continuous and coercive sesqui-linear form, the discrete problem is well posed and enters within the framework of Céa's lemma. We thus have, for all $(\bfphi^h,\bfpsi^h)\in V_0^h\times M^h$,
		\[\norm{(\bfE_\delta - \bfE_\delta^h,\bfF_\delta-\bfF_\delta^h)}{\Hcurl[\Omega]}^2 \le \frac{C}{\alpha} \norm{(\bfE_\delta - \bfphi^h,\bfF_\delta-\bfpsi^h)}{\Hcurl[\Omega]}^2.\]
		Here, $C$ and $\alpha$ are respectively the continuity and coercivity constants of the involved sesqui-linear form. Classical estimates for the considered finite elements yield that the interpolation error is of order $\mcO(h^s)$ (see e.g.~\cite{Monk03}) under the given regularity assumptions on the exact solution. However, the coercivity constant of the sesqui-linear form is given by $\alpha = \min\collection{\delta,1}$. The final discretization error is thus of order $\mcO\left(\frac{h^s}{\min(\delta,1)}\right)$.
	\end{proof}

	\autoref{p:discretization} implies that the regularization parameter $\delta$ has to be chosen with respect to the mesh size and arbitrary small values of $\delta$ are prohibited.

	In all the simulations, we fix $\eps_0 = \mu_0 = \omega = 1$ and $\kappa = 1 + i$. Our reference solution is the plane wave $\bfE(\bfx) = \bfeta^\perp e^{ik\sqrt{\kappa}\bfeta\cdot\bfx}$ where $\bfeta \in \R^2$ is the wave propagation vector and $\bfeta^\perp$ is a unit vector orthogonal to $\bfeta$. The square-root $\sqrt{\kappa}$ stands for the classical complex square-root with branch-cut along the negative real axis.

	\subsubsection{Unit disc}

	In the first example, $\Omega$ is the unit disc discretized with a mesh of size $h = \scnum{2.26e-2}$, ending up with \num{43256} triangles and \num{65134} edges. Two different configurations of accessible/inaccessible parts are tested (cf. \autoref{fig:config_Gamma}):
	\begin{enumerate}
		\item In the configuration G34, $\Gamma_0$ is the arc of circle starting at angle 0 and ending at angle $\dfrac{3\pi}{2}$.  The accessible part then represents \pc{75} of the  boundary $\Gamma$.
		\item In the configuration GE37, $\Gamma_0$ represents a set of 37 equally distributed electrodes of common length $\displaystyle \frac{\pi}{25}$. Electrodes cover \pc{81.6} of the boundary $\Gamma$.
	\end{enumerate}

	\begin{figure}[hbt]
		\centering
		\includegraphics[width=0.3\textwidth]{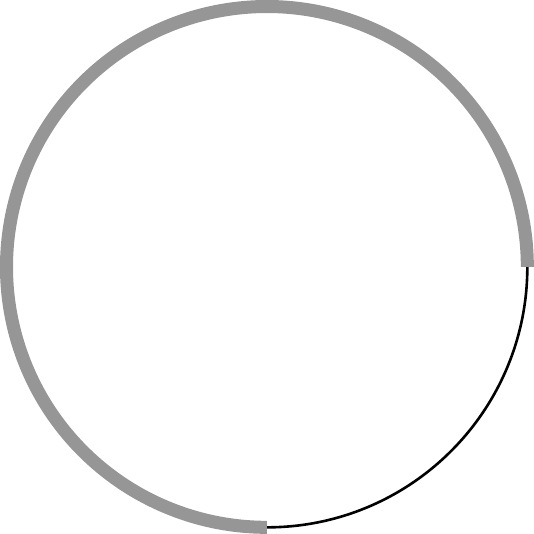}
		\hspace{0.2\textwidth}
		\includegraphics[width=0.3\textwidth]{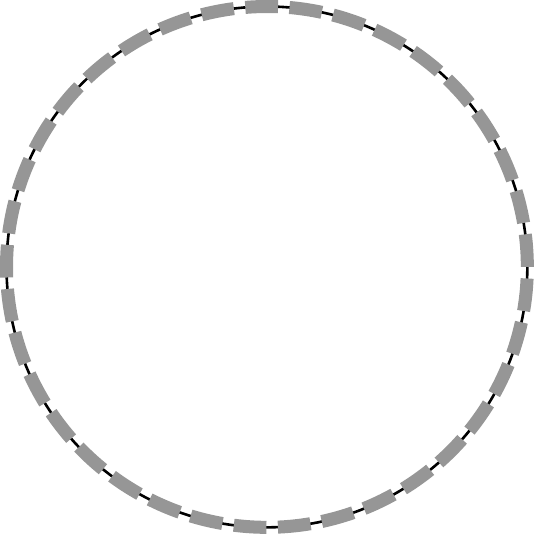}
		\caption{Choice of the accessible part $\Gamma_0$ (grey line). Left: configuration G34. Right: configuration GE37.}
		\label{fig:config_Gamma}
	\end{figure}

	In \autoref{fig:err_delta}, we show the relative error in the $L^2(\Omega)$-norm between the exact solution $\bfE$ and its approximation obtained by the numerical resolution of \eqref{eq:qr}, with respect to $\delta$. For both configurations, we indicate the parameter $\delta$ for which the error is minimal. Notice that the behaviour of the error for the configuration G34 shown in \autoref{fig:err_delta} is in agreement with the error analysis which claims that the error behaves as $\mcO\left(\frac{h^s}{\delta}\right)$: at fixed mesh size the error increases for small values of $\delta$.

	In \autoref{tab:err_qr}, we report different errors obtained for the value $\delta = \num{9.103e-7}$. The approximation of the electric field $\bfE$ is more accurate on the whole domain $\Omega$ than on the inaccessible part $\Gamma_1$. The configuration GE37 yields better results since the  accessible part covers a larger part of the boundary $\Gamma$. In \autoref{fig:err_qr}, we illustrate these results by plotting the modulus of the error $\abs{\bfE - \bfE_\delta}$ in the domain $\Omega$. The largest errors are located on $\Gamma_1$, and in particular at the intersection between $\Gamma_0$ and $\Gamma_1$.
	Indeed, the solution $(\bfE_\delta,\bfF_\delta)$ of \eqref{eq:qr} is the weak solution of a boundary value problem in $\Omega$ with mixed boundary conditions for $\bfE_\delta$ and $\bfF_\delta$. Mixed boundary conditions are known to induce singularities in the fields (see e.g. \cite{Grisvard} for the simpler case of the Laplace operator). This could explain the singular behavior at the intersection points.

	\begin{figure}[hbt]
		\centering
		\includegraphics[width=0.45\textwidth]{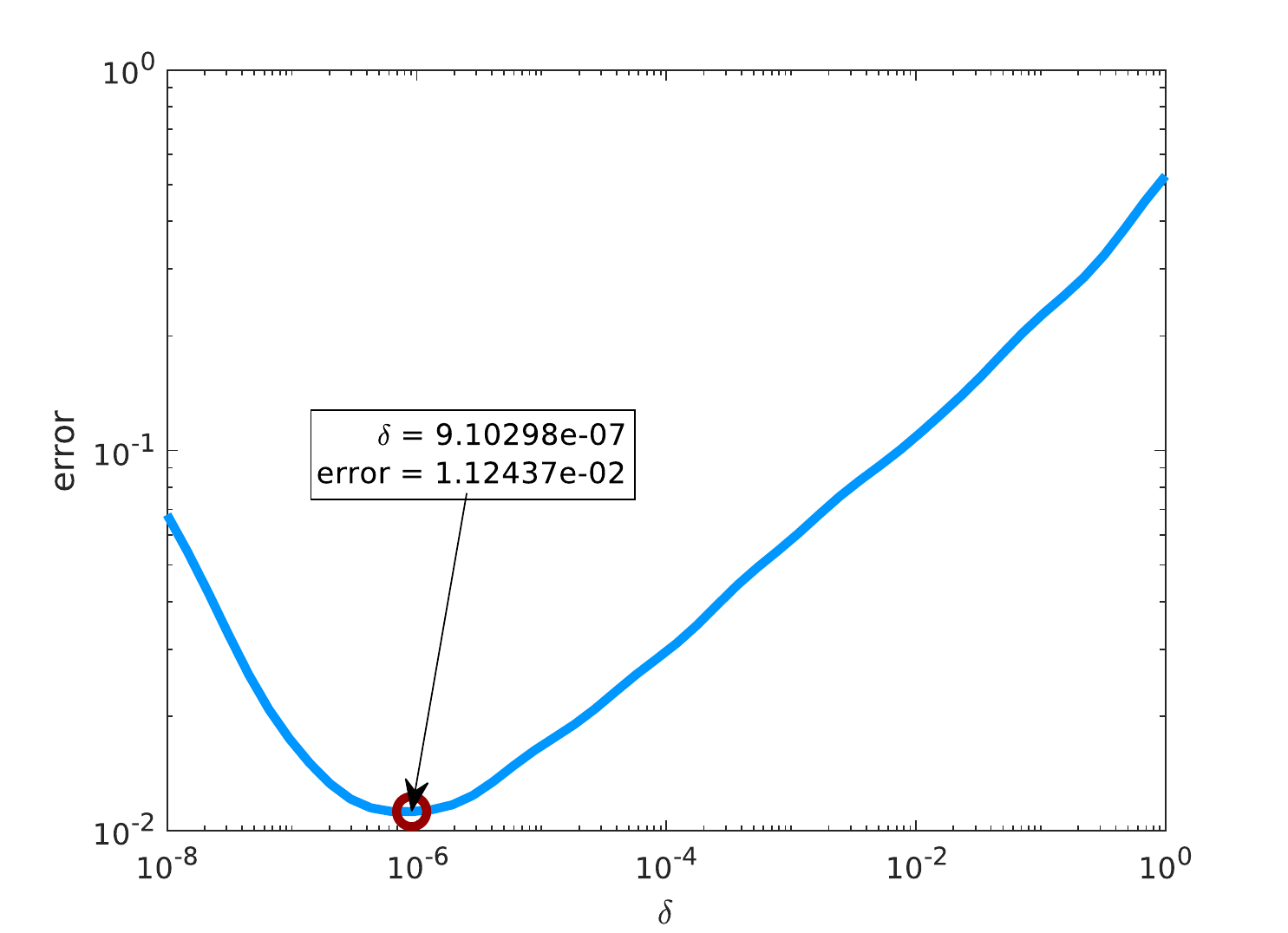}
		\hspace{0.05\textwidth}
		\includegraphics[width=0.45\textwidth]{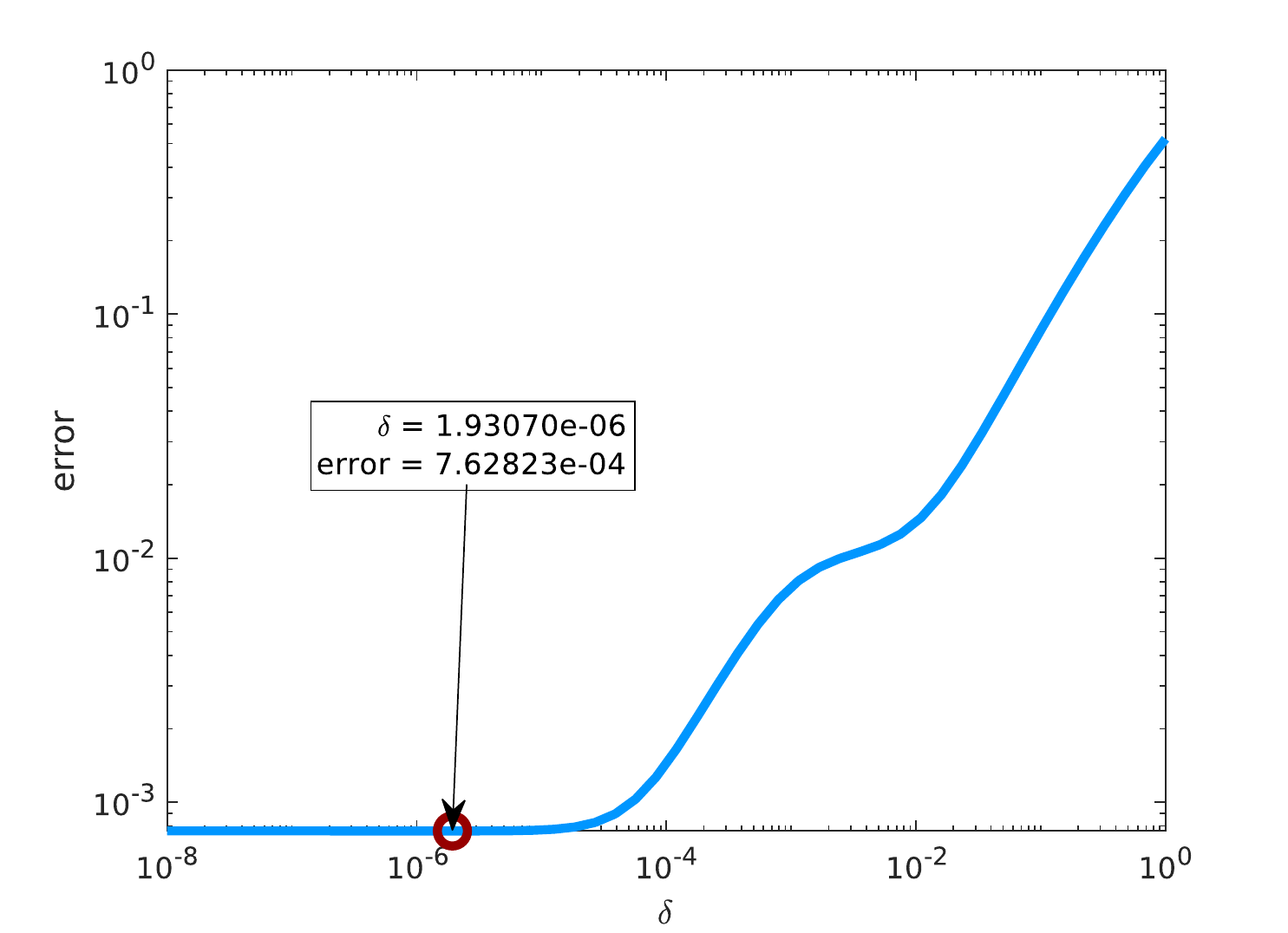}
		\caption{Unit disc. QR method. Relative error $\frac{\norm{\bfE - \bfE_\delta}{0,\Omega}}{\norm{\bfE}{0,\Omega}}$ with respect to the regularization parameter $\delta$. Left: configuration G34. Right: configuration GE37.}
		\label{fig:err_delta}
	\end{figure}

	\begin{table}[hbt]
		\centering
		\caption{Unit disc. QR method. Errors for $\delta = \num{9.103e-7}$.} 	\label{tab:err_qr}
		\begin{tabular}{cccc}
			\toprule
			Configuration & $\frac{\norm{\bfE - \bfE_\delta}{0,\Omega}}{\norm{\bfE}{0,\Omega}}$ & $\frac{\norm{(\bfE - \bfE_\delta) \times \bfn}{0,\Gamma_1}}{\norm{\bfE \times \bfn}{0,\Gamma_1}}$ & $\norm{\bfF_\delta}{0,\Omega}$ \\
			\midrule
			G34 & 1.1244e-02  & 1.7302e-01 & 1.8374e-04 \\
			GE37 & 7.6288e-04  & 1.5185e-02 & 1.8224e-04 \\
			\bottomrule
		\end{tabular}
	\end{table}

	\begin{figure}[hbt]
		\centering
		\includegraphics[width=0.45\textwidth]{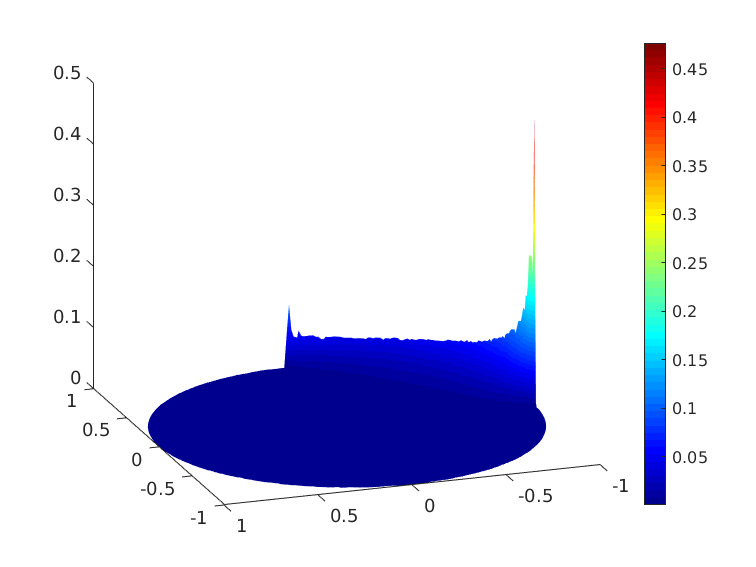}
		\hfill
		\includegraphics[width=0.45\textwidth]{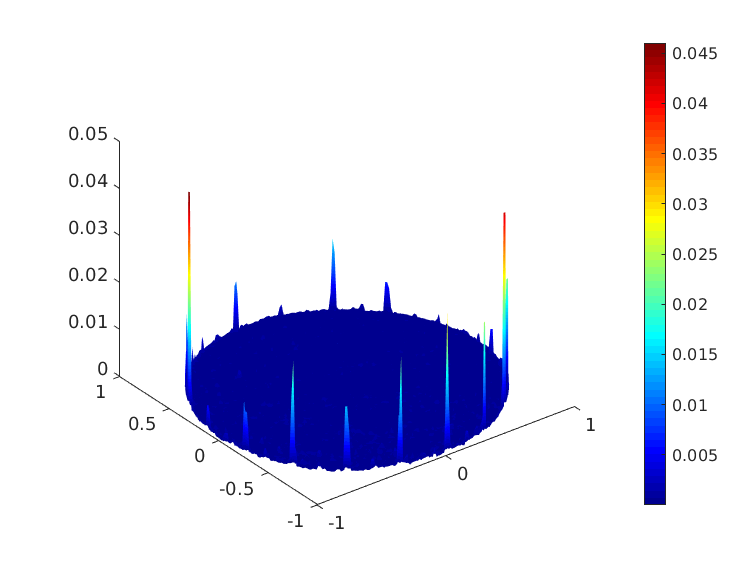}
		\caption{Unit disc. QR method. Modulus of the error $\abs{\bfE - \bfE_{\delta}}$. Left: configuration G34. Right: configuration GE37.}
		\label{fig:err_qr}
	\end{figure}

	\subsubsection{Ring}
	\label{sec:qr-ring}

	For the inverse problems we are studying, we are interested in ring-like domains which, for instance, model the different tissues of the head. The goal consists in mapping  the data measured on part of the exterior boundary to an inner interface, allowing in a second step to use reconstruction methods in the interior domain. In this section, we present numerical results on a ring of internal radius 0.75, discretized with a mesh size $h = \scnum{2.05e-2}$, \num{17919} triangles and \num{27316} edges. Three measurement configurations are compared: G34 and GE37, both described in the previous section, and GExt where $\Gamma_0$ is the whole external boundary (see \autoref{fig:config_Gamma_ring}).

	\begin{figure}[hbt]
		\centering
		\includegraphics[width=0.3\textwidth]{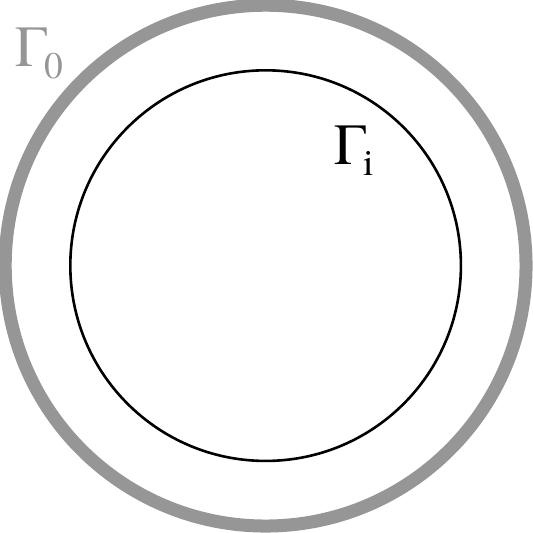}
		\hfill
		\includegraphics[width=0.3\textwidth]{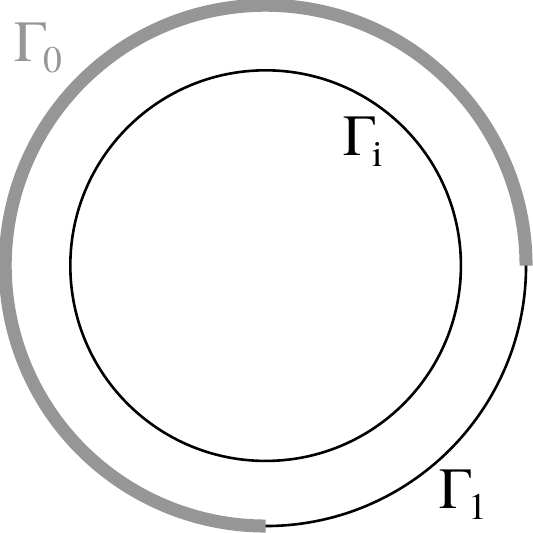}
		\hfill
		\includegraphics[width=0.3\textwidth]{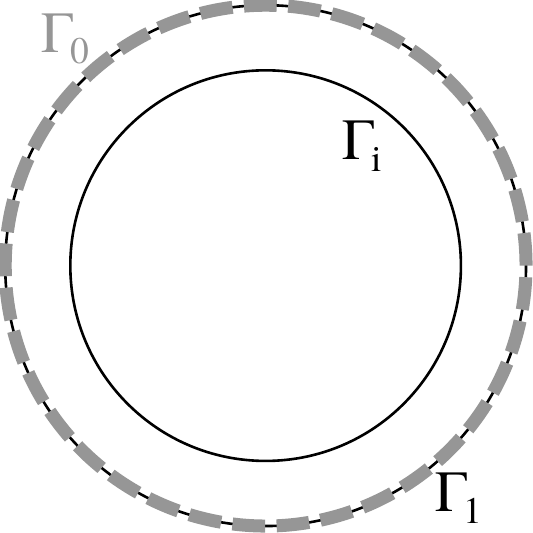}
		\caption{Choice of the accessible part $\Gamma_0$ (grey line).  Left: configuration GExt ($\Gamma_0$ covers \pc{57} of the ring boundary). Middle: configuration G34 (\pc{43}). Right: configuration GE37 (\pc{47}).}
		\label{fig:config_Gamma_ring}
	\end{figure}

	The relative errors with respect to $\delta$ are shown in \autoref{fig:err_delta_ring}. For each configuration, we list in \autoref{tab:err_qr_ring} different errors obtained for the value $\delta$ realizing the minimum indicated in \autoref{fig:err_delta_ring}.
	Notice that for the ring-like domain we split the inaccessible part into an exterior part $\Gamma_1$ and an interior part $\Gamma_i$. Furthermore, the modulus of the error $\abs{\bfE - \bfE_{\delta}}$ is reported in \autoref{fig:err_qr_ring}. We observe that the electric field $\bfE$ is well approximated in the ring by the QR approach when the data are available on the entire exterior boundary (configuration GExt). The transmission of the information on the inner boundary $\Gamma_i$ is very accurate. The configuration GE37 with electrodes leads also to very interesting results with errors below $\pc{4}$ on the inaccessible part $\Gamma_1$ and the inner boundary $\Gamma_i$. The analysis of the results obtained with configuration G34 is less obvious: whereas the error of the auxiliary field $\bfF_\delta$ is very satisfying, we observe on $\Omega$, $\Gamma_1$ and $\Gamma_i$ errors of \pc{35}, \pc{70}, and \pc{48}, respectively. The next section will propose a possible improvement of the QR method for such a configuration.

	\begin{figure}[hbt]
		\centering
		\includegraphics[width=0.32\textwidth]{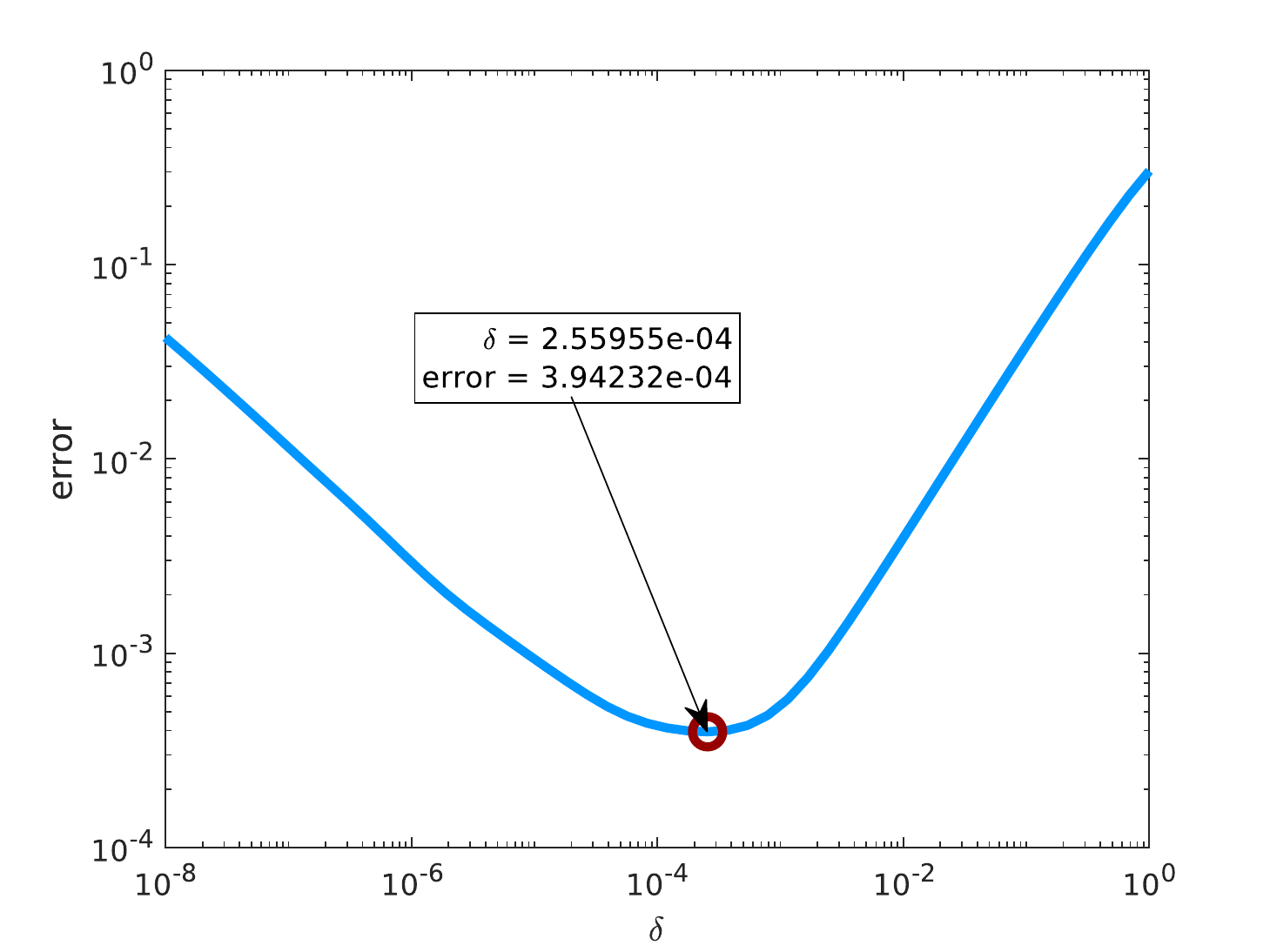}
		\hfill
		\includegraphics[width=0.32\textwidth]{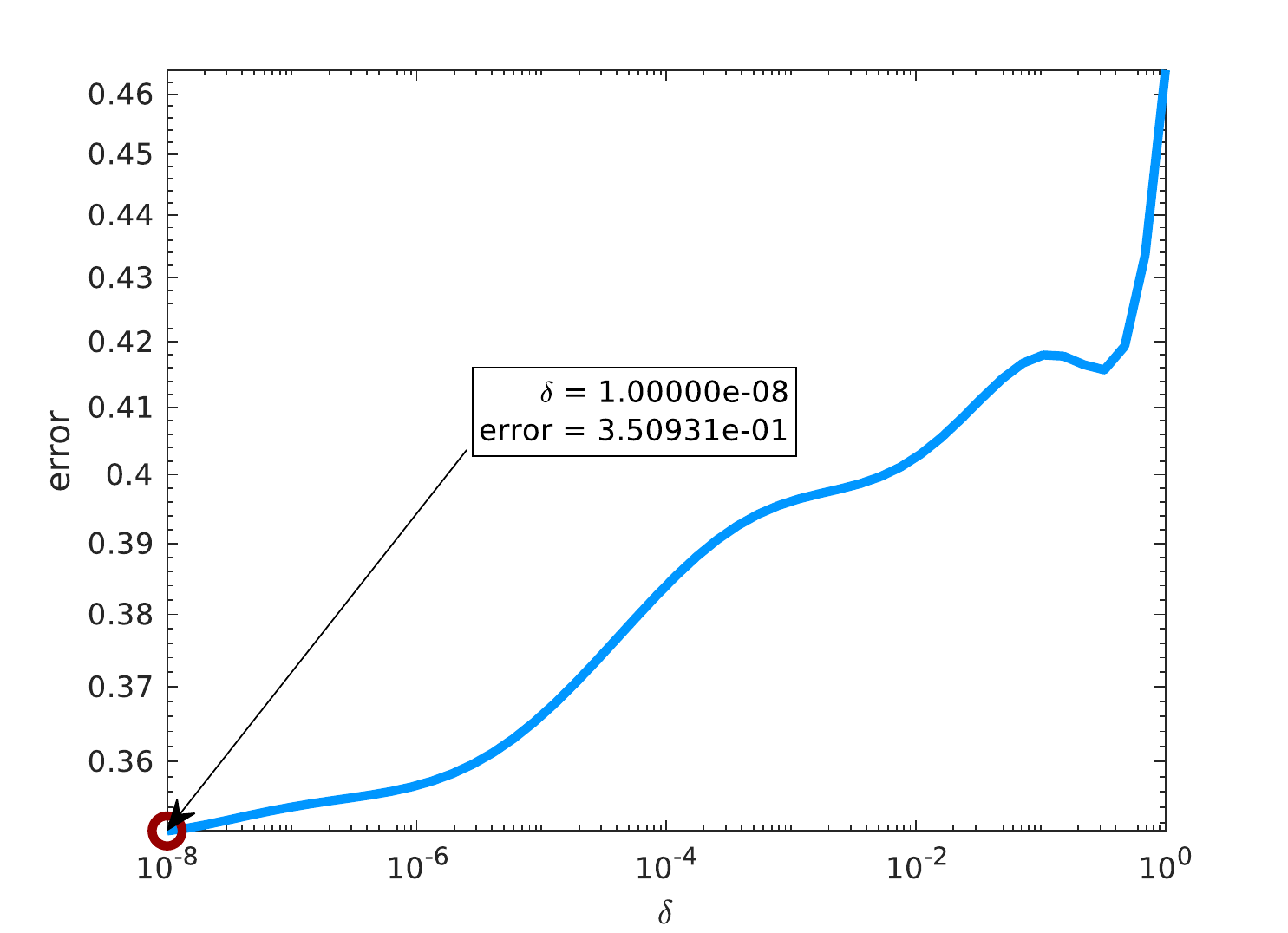}
		\hfill
		\includegraphics[width=0.32\textwidth]{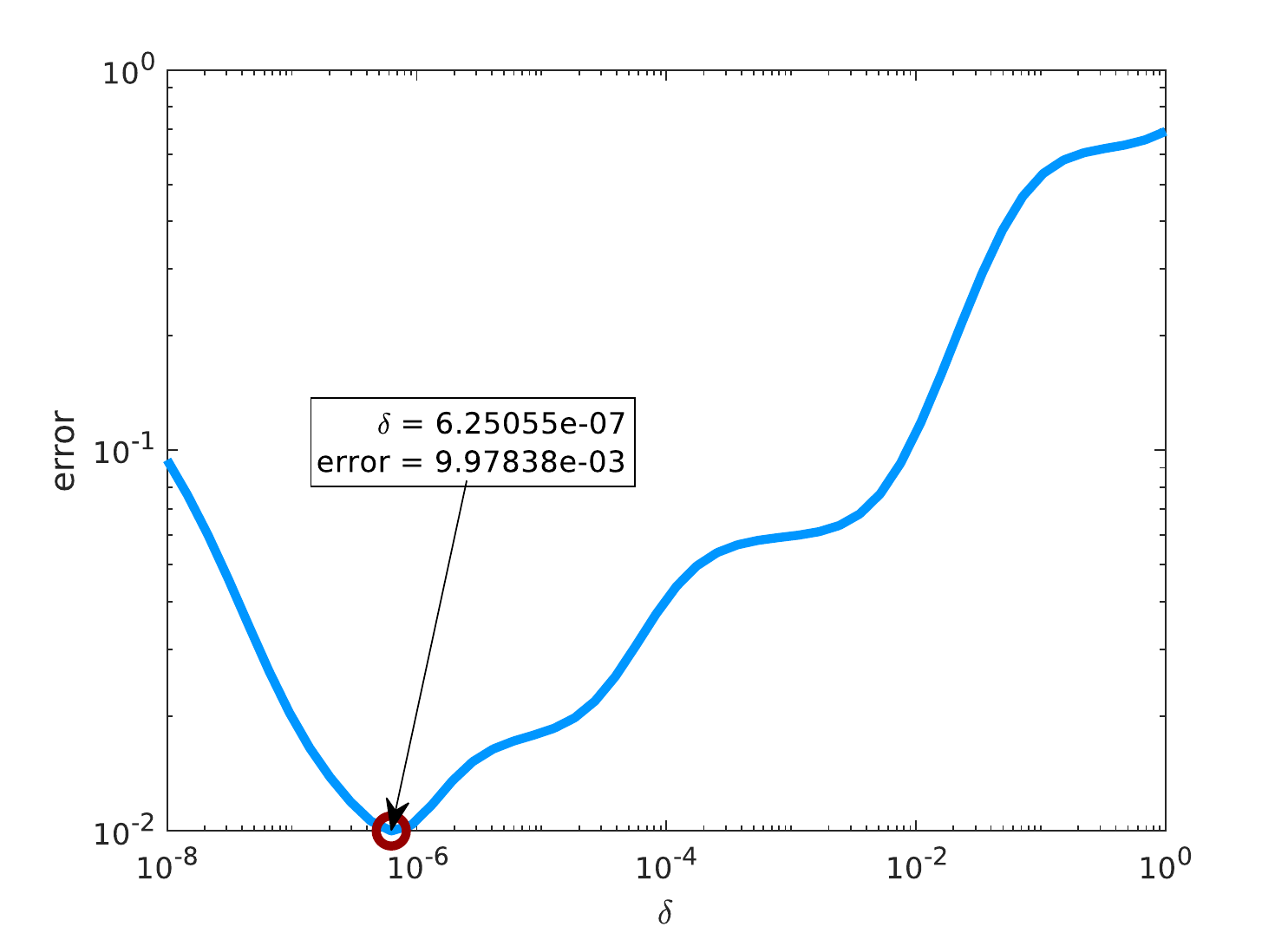}
		\caption{Ring. QR method. Relative error $\frac{\norm{\bfE - \bfE_\delta}{0,\Omega}}{\norm{\bfE}{0,\Omega}}$ with respect to the regularization parameter $\delta$. Left: configuration GExt. Middle: configuration G34. Right: configuration GE37.}
		\label{fig:err_delta_ring}
	\end{figure}

	\begin{table}[hbt]
		\centering
		\caption{Ring. QR method. Errors.}
		\label{tab:err_qr_ring}

		\begin{tabular}{ccccc}
			\toprule
			Configuration & $\frac{\norm{\bfE - \bfE_\delta}{0,\Omega}}{\norm{\bfE}{0,\Omega}}$  & $\frac{\norm{(\bfE - \bfE_\delta) \times \bfn}{0,\Gamma_1}}{\norm{\bfE \times \bfn}{0,\Gamma_1}}$ & $\frac{\norm{(\bfE - \bfE_\delta) \times \bfn}{0,\Gamma_\text{i}}}{\norm{\bfE \times \bfn}{0,\Gamma_\text{i}}}$ & $\norm{\bfF_\delta}{0,\Omega}$ \\
			\midrule
			GExt & 3.9423e-04 & 6.6661e-04 & 6.6661e-04 & 1.7964e-04 \\
			G34 & 3.5093e-01  & 6.9598e-01 & 4.8011e-01 & 8.1894e-05 \\
			GE37 & 9.9784e-03  & 3.9068e-02 & 3.8737e-02 & 8.4337e-05 \\
			\bottomrule
		\end{tabular}
	\end{table}

	\begin{figure}[hbt]
		\centering
		\includegraphics[width=0.32\textwidth]{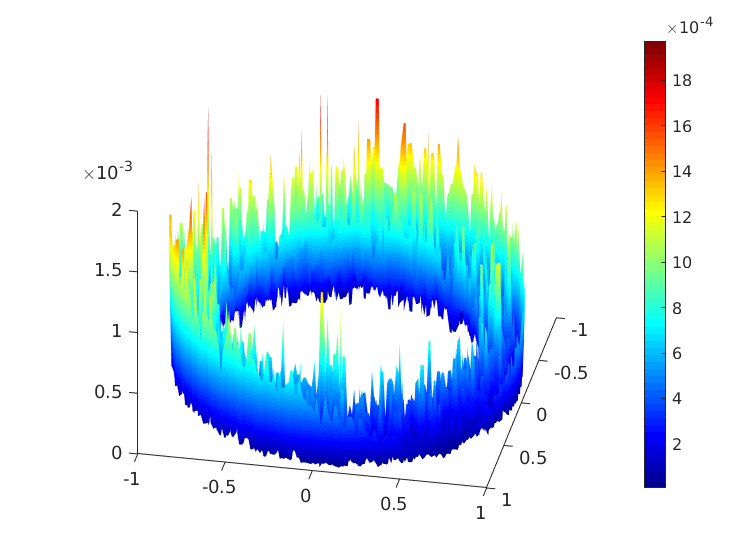}
		\hfill
		\includegraphics[width=0.32\textwidth]{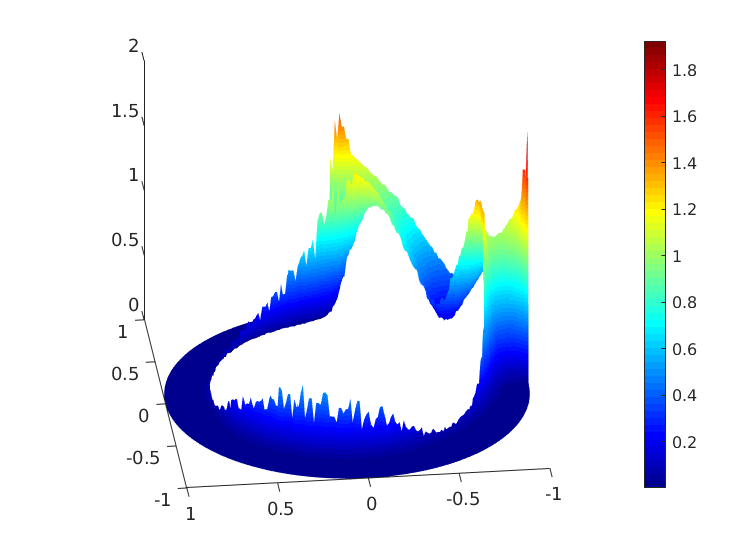}
		\hfill
		\includegraphics[width=0.32\textwidth]{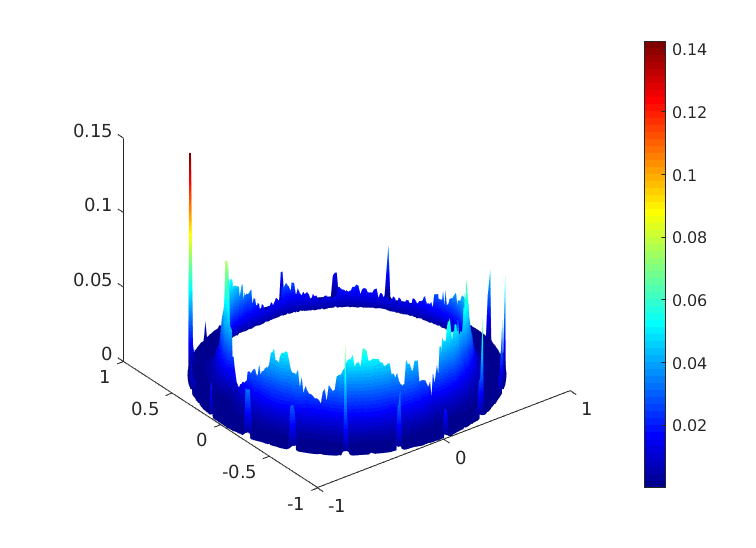}
		\caption{Ring. QR method. Modulus of the error $\abs{\bfE - \bfE_\delta}$. Left: configuration GExt. Middle: configuration G34. Right: configuration GE37.}
		\label{fig:err_qr_ring}
	\end{figure}

	\subsubsection{Extension of the computational domain}

	For similar configurations of the accessible part, the errors obtained on the disc are more accurate than on the ring. This leads us to the idea of extending the computational domain to a full disc $\tilde{\Omega} = \left(\overline{\Omega} \cup \overline{\Omega_\text{int}}\right)^\circ$ before restricting the solution to the initial ring $\Omega$ through the \autoref{algo:extension-restriction}.

	\begin{algorithm}
		\caption{Extension/restriction method}
		\label{algo:extension-restriction}

		\KwIn{$(\bff,\bfg)$ on $\Gamma_0$}
		\KwOut{Solution $({\bfE}_\delta, {\bfF}_\delta)$ of the QR problem on the ring $\Omega$}
		\begin{enumerate}
			\item Extend the coefficients $\eps$ and $\sigma$ defined on the ring $\Omega$ to admissible coefficients $\tilde{\eps}$ and $\tilde{\sigma}$
			defined on $\tilde{\Omega}$.\;
			\item Compute $(\tilde{\bfE}_\delta,\tilde{\bfF}_\delta)$, solution to the QR method on $\tilde{\Omega}$.
			\item Define the fields $\bfE_\delta \coloneqq \tilde{\bfE}_{\delta\vert \Omega}$ and $\bfF_\delta \coloneqq \tilde{\bfF}_{\delta\vert \Omega}$.
		\end{enumerate}
	\end{algorithm}

	From a theoretical point of view, if the data $(\bff,\bfg)$ belong to the Cauchy data set $C(\tilde{\eps},\tilde{\sigma};\Gamma_0)$ with respect to the extended domain $\tilde{\Omega}$, the sequence $(\tilde{\bfE}_\delta,\tilde{\bfF}_\delta)$ converges to $(\tilde{\bfE},0)$ where $\tilde{\bfE}$ is the solution of  the Cauchy problem~\eqref{eq:Cauchy} in $\tilde{\Omega}$ associated with the coefficients $\tilde{\eps}$ and $\tilde{\sigma}$. The restriction $\bfE$ to the ring $\Omega$ satisfies the Cauchy problem with data $(\bff,\bfg)$ on $\Omega$. Thanks to the unique continuation principle, $\bfE$ is the only possible solution and coincides with the limit of the sequence obtained by the QR method applied on $\Omega$.

	\autoref{fig:err_delta_filled_ring} shows the error with respect to $\delta$ obtained by the extension/restriction method for the previous three configurations. The errors obtained with the values of $\delta$ realizing the minima are reported in \autoref{tab:err_qr_filled_ring} and illustrated in \autoref{fig:err_qr_filled_ring}. We notice a significative improvement of the approximation on the inaccessible part: \pc{12.5} instead of \pc{70} on $\Gamma_1$, \pc{0.8} instead of \pc{48} on $\Gamma_{i}$. The drawback lies in the increasing computational cost since the extended domain leads to a larger number of unknowns.

	The numerical results of the previous sections attest the efficiency of the QR method for Maxwell's equations in the case of compatible data belonging to the trace space.

	\begin{remark}
		\label{rem:extension-restriction}
		It is important to notice that the extension/restriction method is limited to the case where the belonging of the data $(\bff,\bfg)$ to the Cauchy set with respect of the extended domain can be guaranteed in addition to the usual assumption that the data belong to the Cauchy set of the original domain.
		In the present simulations, the (synthetic) data on the exterior boundary have been generated by solving the Maxwell equations on the whole disc. The QR-method has then been applied to the ring in order to perform data completion on the interior boundary. Hence, the data belong to the Cauchy set on both the ring and the extended domain.
		In a general context, the question of the choice of the extended domain needs more discussion. While choosing $\tilde{\Omega}$ as the whole disc seems natural, it raises an issue: the Cauchy problem is not guaranteed to admit a solution in this domain.
		However, we can see the extension/restriction method from a different point of view in which the existence of a solution of the Cauchy problem in the extended domain holds. In the practical applications evoked in \autoref{sec:qr-ring} as for instance the cortical mapping, it is realistic to assume that the refractive index is known in a neighborhood of the boundary, denoted by $\tilde{\Omega}$ (typically the scalp layer for medical applications), and also that a solution of the Cauchy problem in the whole disc exists (corresponding to the measurements $({\bf f}, {\bf g})$). However, the electromagnetic coefficients are unknown in the inner part of the disc. Extending the coefficients from the known part to the inner part in an arbitrary way could lead to a situation where the data do no longer belong to the corresponding Cauchy set, except the (unrealistic) situation where the extended coefficients coincide exactly with the unknown	parameters. But in our application, it is sufficient
		 to solve the data completion problem in a ring $\Omega$ which represents only a part of the known neighborhood $\tilde{\Omega}$, and we apply the QR method for this issue. The idea is to
		extend the QR method to the ring $\tilde{\Omega}$, and to restrict the solution to the thinner ring $\Omega$. In this case, a solution of the Cauchy
		problem exists in $\tilde{\Omega}$ and $\Omega$ and the extension/restriction method still yields better results than the classical QR-method on $\tilde{\Omega}$. Such a configuration has been tested successfully in \cite[Section~6.7.3, pp.~162--164]{PHDHeleine19}.
	\end{remark}

	\begin{figure}[hbt]
		\centering
		\includegraphics[width=0.32\textwidth]{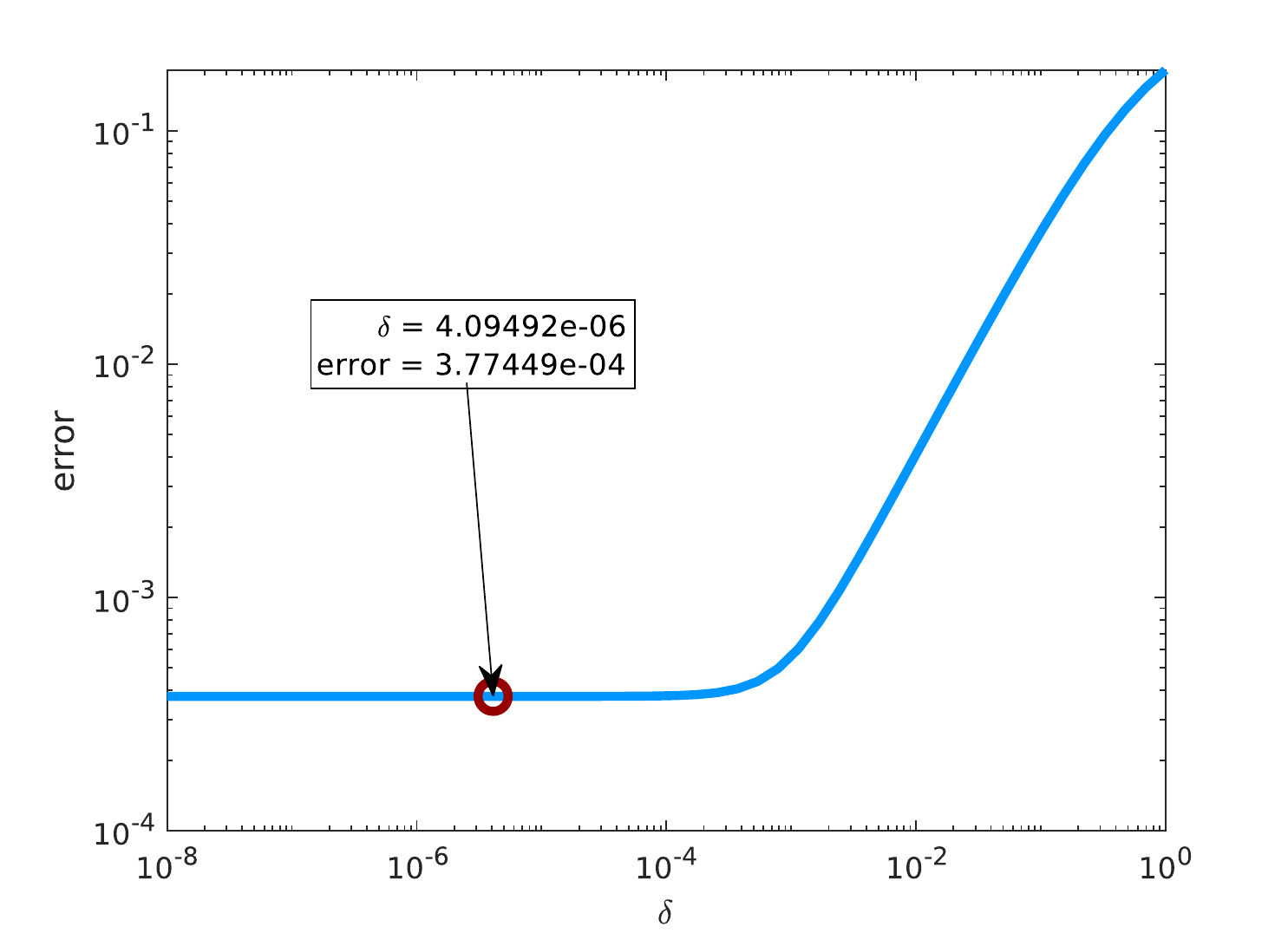}
		\hfill
		\includegraphics[width=0.32\textwidth]{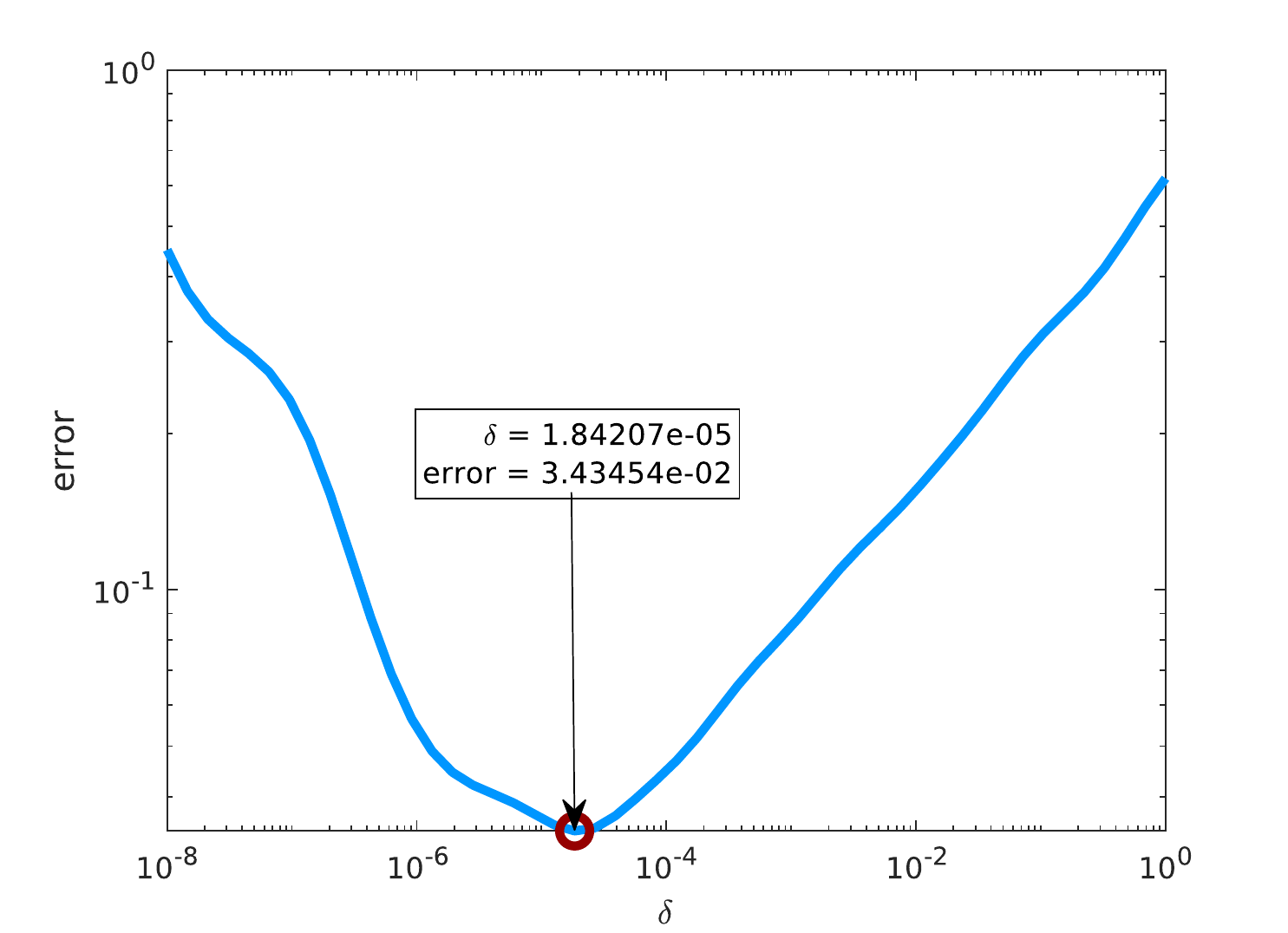}
		\hfill
		\includegraphics[width=0.32\textwidth]{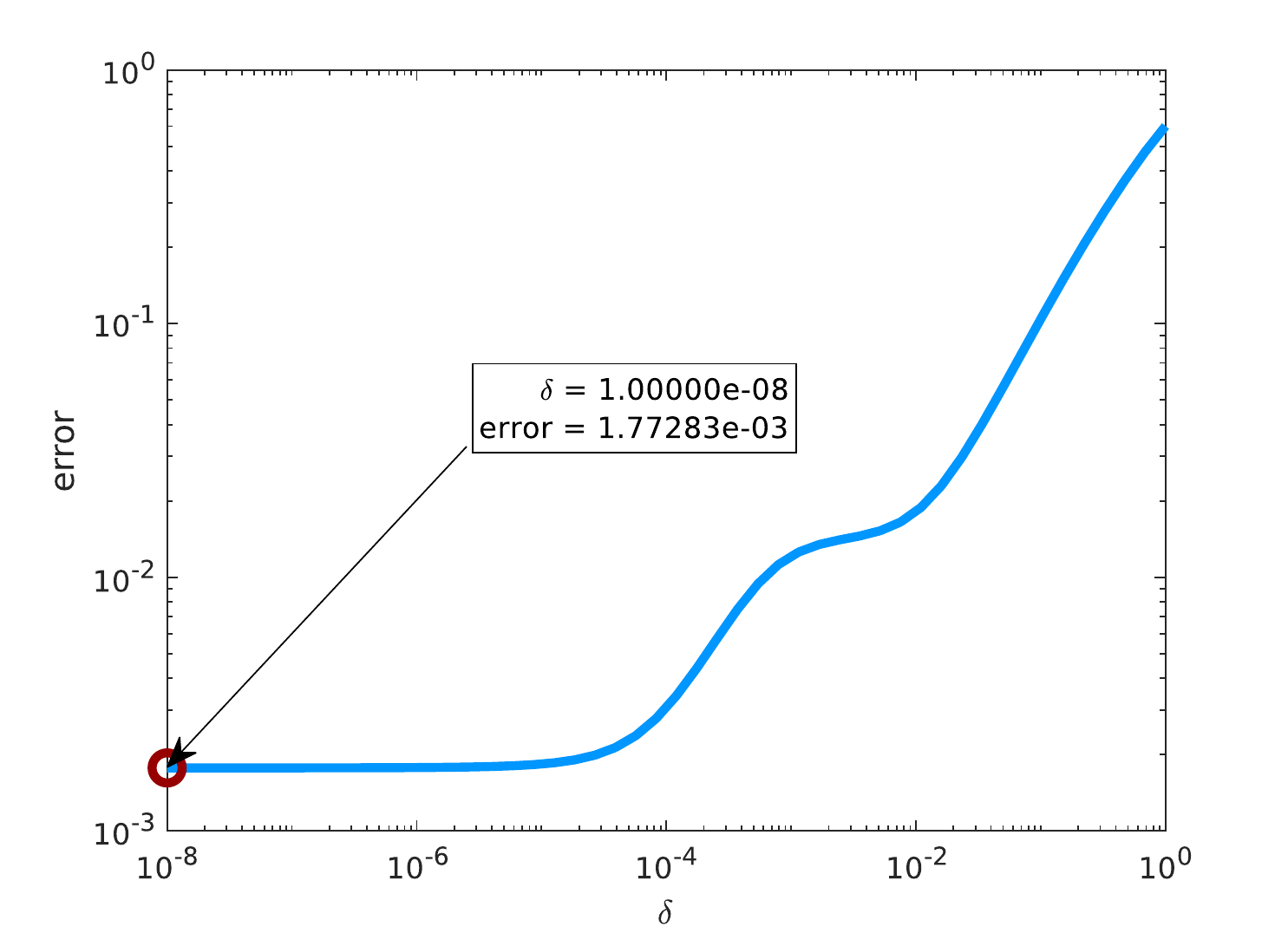}
		\caption{Ring. QR method. Extension/restriction method. Relative error $\frac{\norm{\bfE - \bfE_\delta}{0,\Omega}}{\norm{\bfE}{0,\Omega}}$ with respect to the regularization parameter $\delta$. Left: configuration GExt. Middle: configuration G34. Right: configuration GE37.}
		\label{fig:err_delta_filled_ring}
	\end{figure}

	\begin{table}[hbt]
		\centering
		\caption{Ring. QR method. Extension/restriction method. Errors.}
		\label{tab:err_qr_filled_ring}

		\begin{tabular}{ccccc}
			\toprule
			Configuration & $\frac{\norm{\bfE - \bfE_\delta}{0,\Omega}}{\norm{\bfE}{0,\Omega}}$ & $\frac{\norm{(\bfE - \bfE_\delta) \times \bfn}{0,\Gamma_1}}{\norm{\bfE \times \bfn}{0,\Gamma_1}}$ & $\frac{\norm{(\bfE - \bfE_\delta) \times \bfn}{0,\Gamma_\text{i}}}{\norm{\bfE \times \bfn}{0,\Gamma_\text{i}}}$ & $\norm{\bfF_\delta}{0,\Omega}$ \\
			\midrule
			GExt & 3.7745e-04  & 1.5665e-04 & 1.5665e-04 & 3.1125e-04 \\
			G34 & 3.4345e-02  & 1.2558e-01 & 8.1477e-03 & 2.6564e-04 \\
			GE37 & 1.7728e-03  & 1.1968e-02 & 2.9520e-04 & 9.0131e-05 \\
			\bottomrule
		\end{tabular}
	\end{table}

	\begin{figure}[hbt]
		\centering
		\includegraphics[width=0.32\textwidth]{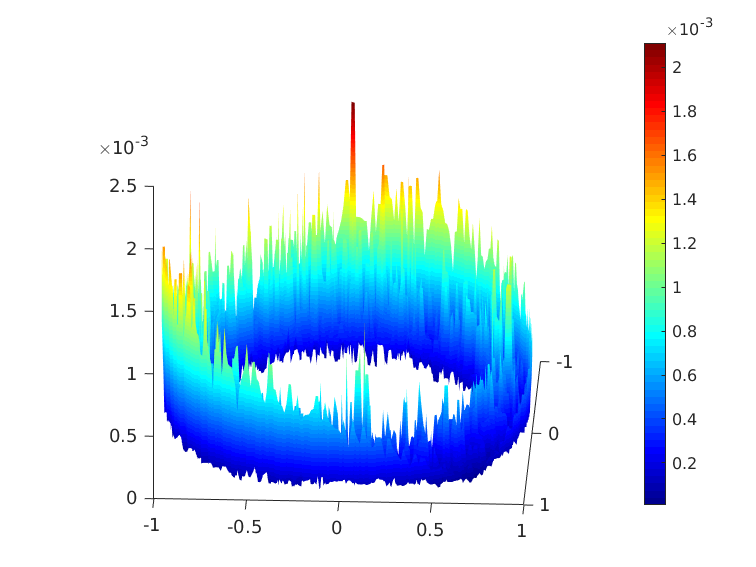}
		\hfill
		\includegraphics[width=0.32\textwidth]{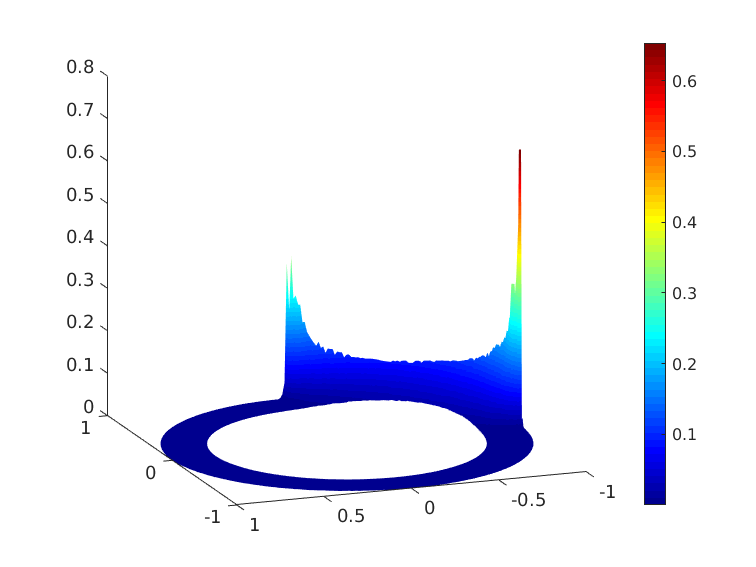}
		\hfill
		\includegraphics[width=0.32\textwidth]{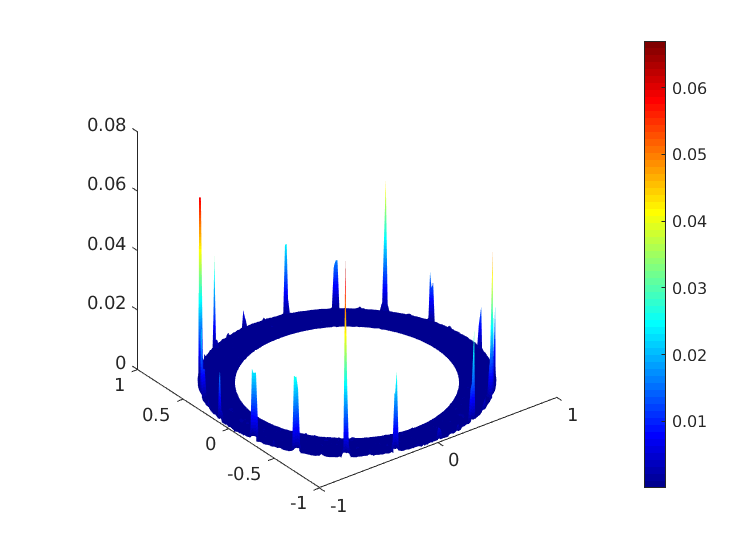}
		\caption{Ring. QR method. Extension/restriction method. Modulus of the error $\abs{\bfE - \bfE_\delta}$.  Left: configuration GExt. Middle: configuration G34. Right: configuration GE37.}
		\label{fig:err_qr_filled_ring}
	\end{figure}

	\section{Quasi-reversibility formulations for noisy data}
	\label{sec:RRQR}

	Noisy data will, in general, not belong to the trace space $Y(\Gamma_0)$. Consequently, problem \eqref{eq:qr} is not well defined since no lifting exists for the data $\bff$. The boundary condition $\gamma_t(\bfE) = \bff$ has thus to be enforced weakly. Moreover, if $\bfg\not\in Y(\Gamma_0)$, the definition of the linear form $\ell(\cdot)$ on the vector space $M$ is no longer straightforward. To overcome these difficulties, we modify in this section the basic vector space and assume in the sequel that the boundary data $(\bff,\bfg)$ belong to $L^2(\Gamma_0)^3 \times L^2(\Gamma_0)^3$. Notice that this assumption does not imply that the data belong to the trace space.

	\subsection{A first version}
	\label{ss-relax1}

	Consider the vector space
	\begin{equation}
		\label{eq:V}
		V = \set*{\bfv \in \Hcurl[\Omega]}{\gamma_t(\bfv) \in L^2(\Gamma_0)^3}
	\end{equation}
	with the norm
	\begin{equation}
		\label{eq:normV}
		\norm{\bfv}{V} = \left(\norm{\bfv}{\Hcurl[\Omega]}^2 + \norm{\gamma_t(\bfv)}{0,\Gamma_0}^2\right)^{1/2}.
	\end{equation}
	The space $M$ will now be defined as the following subspace of $V$,
	\begin{equation}
		\label{eq:M}
		M = \set{\bfv \in V}{\gamma_t(\bfv) = 0 \stext{on} \Gamma_1},
	\end{equation}
	equipped with the norm of $V$, $\norm{}{V}$. Since we cannot assume that there is a lifting of the boundary data $\bff$ in $V$, the boundary condition will be imposed weakly through a penalization term. The relaxed version of the mixed problem then reads
	\begin{equation}
		\label{eq:qr-relaxed}
		\left\{
		\begin{array}{rcl@{\hspace{4\tabcolsep}}l}
			\multicolumn{4}{l}{\stext[r]{Find} (\bfE_\alpha,\bfF_\alpha) \in V \times M  \stext[l]{such that}} \\
			\multicolumn{1}{l}{\delta\dotprod{\bfE_\alpha}{\bfphi}{V} + \eta^2\dotprod{\gamma_t(\bfE_\alpha)}{\gamma_t(\bfphi)}{0,\Gamma_0}} & & & \multirow{2}{*}{$\forall \bfphi \in V,$} \\
			+ a(\bfphi,\bfF_\alpha) &=& \eta^2\dotprod{\bff}{\gamma_t(\bfphi)}{0,\Gamma_0}, & \\
			a(\bfE_\alpha,\bfpsi) - \dotprod{\bfF_\alpha}{\bfpsi}{V} &=& \ell(\bfpsi), & \forall \bfpsi \in M,
		\end{array}
		\right.
	\end{equation}
	where the subscript $\alpha = (\delta,\eta)$ indicates that the solution of \eqref{eq:qr-relaxed} depends on the regularization parameter $\delta > 0$ and the relaxation parameter $\eta > 0$.

	\begin{theorem}
		\label{thm:convergence-relaxed}
		Let $(\bff,\bfg)\in L^2(\Gamma_0) \times L^2(\Gamma_0)$. For any $\alpha = (\delta,\eta)$ such that $\delta > 0$ and $\eta > 0$, problem~\eqref{eq:qr-relaxed} admits a unique solution $(\bfE_\alpha,\bfF_\alpha) \in V \times M$.

		If, in addition, $(\bff,\bfg)$ belongs to the Cauchy data set $C(\eps, \sigma; \Gamma_0)$, then
		\begin{equation}
			\label{eq:convergence-relaxed}
			\lim_{\delta \to 0} (\bfE_\alpha,\bfF_\alpha) = (\bfE, 0)
		\end{equation}
		in $V \times M$ for any fixed $\eta > 0$. Here, $\bfE$ is the unique solution of the Cauchy problem \eqref{eq:Cauchy}.
		The following estimates hold true
		\begin{align}
			\label{eq:estim-relaxed-1}
			\norm{\bfF_\alpha}{V} &\leq \sqrt{\delta}\norm{\bfE}{V}\ \forall \eta > 0, \\
			\label{eq:estim-relaxed-2}
			\norm{\gamma_t\bfE_\alpha - \bff}{0,\Gamma_0} &\leq \frac{\sqrt{\delta}}{\eta} \norm{\bfE}{V}.
		\end{align}
	\end{theorem}

	The estimates~\eqref{eq:estim-relaxed-1} and \eqref{eq:estim-relaxed-2} in \autoref{thm:convergence-relaxed} give further insight in the behavior of the sequence $(\bfE_\alpha,\bfF_\alpha)$. Indeed, \eqref{eq:estim-relaxed-1} yields a convergence order for the convergence of $\bfF_\alpha$ to $0$ independently from the choice of $\eta$. Estimate \eqref{eq:estim-relaxed-2} suggests that we should choose $\eta$ such that $\frac{\sqrt{\delta}}{\eta}\to 0$. Large values of $\eta$ should accelerate the convergence of the trace of $\bfE_\alpha$ on $\Gamma_0$.

	\begin{proof}
		Problem~\eqref{eq:qr-relaxed} can be written in the following closed form,
		\[
			\left\{
			\begin{array}{l}
				\stext[r]{Find} (\bfE_\alpha,\bfF_\alpha) \in V \times M \stext[l]{such that} \\
				A_\alpha\left((\bfE_\alpha,\bfF_\alpha), (\bfphi,\bfpsi)\right) = L_\alpha\left((\bfphi,\bfpsi)\right)\ \forall (\bfphi,\bfpsi) \in V \times M,
			\end{array}
			\right.
		\]
		where the sesqui-linear form $A_\alpha(\cdot,\cdot)$ and the linear form $L_\alpha(\cdot)$ are defined on $V \times M$ by
		\[
			A_\alpha\left((\bfu,\bfv), (\bfphi,\bfpsi)\right) = \delta\dotprod{\bfu}{\bfphi}{V} + \eta^2\dotprod{\gamma_t(\bfu)}{\gamma_t(\bfphi)}{0,\Gamma_0} + a(\bfphi,\bfv) - a(\bfu,\bfpsi) + \dotprod{\bfv}{\bfpsi}{M}
		\]
		and
		\[
			L_\alpha\left((\bfphi,\bfpsi)\right) = \eta^2\dotprod{\bff}{\gamma_t(\bfphi)}{0,\Gamma_0} - \ell(\bfpsi).
		\]
		The continuity of $A_\alpha(\cdot,\cdot)$ and $L_\alpha(\cdot)$ is obvious, and coercivity follows since
		\[
			A_\alpha\left((\bfu,\bfv),(\bfu,\bfv)\right) = \delta\norm{\bfu}{V}^2 + \eta^2\norm{\gamma_t(\bfu)}{0,\Gamma_0}^2 + \norm{\bfv}{V}^2 \geq \min(\delta,1) \norm{(\bfu,\bfv)}{V \times V}^2
		\]
		taking into account that the scalar product in $M$ is the one of $V$. We thus can apply Lax-Milgram's Theorem to prove existence and uniqueness of a solution of problem~\eqref{eq:qr-relaxed}.

		According to \autoref{p:weakCauchy} and since the new space $M$ is a subspace of the former, the solution $\bfE$ of the Cauchy problem satisfies
		\[a(\bfE,\bfpsi) = \ell(\bfpsi)\ \forall \bfpsi \in M.\]
		Together with the second equation of problem~\eqref{eq:qr-relaxed}, we get
		\begin{equation}
			\label{eq:tmp1}
			a(\bfE_\alpha - \bfE,\bfpsi) - \dotprod{\bfF_\alpha}{\bfpsi}{V} = 0\ \forall \bfpsi \in M.
		\end{equation}
		Now, take $\bfphi = \bfE_\alpha - \bfE$ in problem~\eqref{eq:qr-relaxed} and $\bfpsi = \bfF_\alpha$ in \eqref{eq:tmp1} and substract the latter from the first one. Taking into account that $\gamma_t(\bfE) = \bff$, this yields the fundamental relation
		\begin{equation}
			\label{eq:tmp2}
			\delta\dotprod{\bfE_\alpha}{\bfE_\alpha - \bfE}{V} + \eta^2\norm{\gamma_t(\bfE_\alpha) - \bff}{0,\Gamma_0}^2 + \norm{\bfF_\alpha}{V}^2 = 0
		\end{equation}
		which implies
		\begin{equation}
			\label{eq:tmp2bis}
			\delta\norm{\bfE_\alpha}{V}^2 + \eta^2\norm{\gamma_t(\bfE_\alpha) - \bff}{0,\Gamma_0}^2 + \norm{\bfF_\alpha}{V}^2 \leq \delta\norm{\bfE_\alpha}{V}\norm{\bfE}{V}.
		\end{equation}
		From \eqref{eq:tmp2bis}, we see that $\norm{\bfE_\alpha}{V} \leq \norm{\bfE}{V}$, i.e. the sequence $(\bfE_\alpha)_\alpha$ is bounded with respect to the two parameters $\delta$ and $\eta$. We then deduce from \eqref{eq:tmp2bis} estimation~\eqref{eq:estim-relaxed-1},
		\[\norm{\bfF_\alpha}{V} \leq \sqrt{\delta}\norm{\bfE}{V},\]
		and thus the convergence of $(\bfF_\alpha)_\alpha$ to $0$ whenever the regularization parameter $\delta$ tends to $0$.

		The boundary term can be estimated in a similar way,
		\[\norm{\gamma_t(\bfE_\alpha) - \bff}{0,\Gamma_0} \leq \frac{\sqrt{\delta}}{\eta}\norm{\bfE}{V},\]
		which yields estimation~\eqref{eq:estim-relaxed-2}. In particular, the sequence $(\gamma_t(\bfE_\alpha))_\alpha$ tends to $\bff$ if $\lim\limits_{\delta \to 0} \left(\frac{\sqrt{\delta}}{\eta}\right) = 0 $ which is, for example, the case for any fixed $\eta > 0$.

		It remains to prove the convergence of the sequence $(\bfE_\alpha)_\alpha$. Recall that the sequence is bounded in $V$ which is a Hilbert space. Therefore, there is a subsequence of $(\bfE_\alpha)_\alpha$ that converges weakly in $V$ to a limit field $\tilde{\bfE}$.

		Passing to the limit in the second equation of \eqref{eq:qr-relaxed} yields
		\[a(\tilde{\bfE},\bfpsi) = \ell(\bfpsi)\ \forall \bfpsi \in M\]
		if $\delta \to 0$. Moreover, we get on the one hand
		\[\dotprod{\gamma_t(\bfE_\alpha)}{\xi}{0,\Gamma_0} \to \dotprod{\gamma_t(\tilde{\bfE})}{\xi}{0,\Gamma_0}\ \forall \xi \in L^2(\Gamma_0)^3,\ \xi \cdot \bfn = 0 \stext{on} \Gamma_0\]
		from the weak convergence of $(\bfE_\alpha)_\alpha$ in $V$ and the density of $Y(\Gamma_0) \cap L^2(\Gamma_0)^3$ in the subspace of tangential fields of $L^2(\Gamma_0)^3$, and on the other
		\[\gamma_t(\bfE_\alpha) \to \bff\]
		strongly in $L^2(\Gamma_0)^3$ according to \eqref{eq:estim-relaxed-2} and the assumptions on the parameter set $\alpha$. Consequently, the limit field $\tilde{\bfE}$ satisfies $\gamma_t(\tilde{\bfE}) = \bff$ and is a solution of the weak Cauchy problem. The uniqueness of this solution yields $\tilde{\bfE} = \bfE$.

		We then deduce from \eqref{eq:tmp2} that
		\begin{align*}
			\norm{\bfE_\alpha - \bfE}{V}^2 &= \real{\dotprod{\bfE_\alpha}{\bfE_\alpha - \bfE}{V} - \dotprod{\bfE}{\bfE_\alpha - \bfE}{V}} \\
			&= -\frac{\eta^2}{\delta}\norm{\gamma_t(\bfE_\alpha-\bfE)}{0,\Gamma_0}^2 - \frac{1}{\delta}\norm{\bfF_\alpha}{V}^2 - \real{\dotprod{\bfE}{\bfE_\alpha - \bfE}{V}} \\
			&\leq - \real{\dotprod{\bfE}{\bfE_\alpha}{V}} + \norm{\bfE}{V}^2.
		\end{align*}
		Since $(\bfE_\alpha)_\alpha$ converges weakly to $\bfE$, the above inequality implies the strong convergence, at least for a subsequence. It follows from a standard argument that the whole sequence converges strongly to $\bfE$ which completes the proof.
	\end{proof}

	\begin{remark}
		Under the regularity assumptions of \autoref{p:discretization}, we have an a priori error estimate of order $\frac{h^s}{\min(\delta,1)}$ since the coercivity constant of the bilinear form $A_\alpha(\cdot,\cdot)$ is still given by $\min(\delta,1)$ and the additional boundary term in the $V$-norm can be estimated with help of the continuity of the trace operator.
	\end{remark}

	\subsection{A second version}
	\label{ss-relax2}

	The choice of the space $V$ in the preceding section has been motivated by the penalization of the boundary condition $\bfE \times \bfn = \bff$ which can no longer be imposed strongly if the data do not belong to the trace space $Y(\Gamma_0)$. One may ask however if it is judicious to restrict the belonging of the fields trace to $L^2$ on the accessible part $\Gamma_0$ only. It seems thus natural to investigate another choice for the basic vector space: let
	\begin{equation}
		\label{eq:W}
		W \coloneqq \set*{\bfv \in \Hcurl[\Omega]}{\gamma_t(\bfv) \in L^2(\Gamma)^3}
	\end{equation}
	with the norm
	\begin{equation}
		\label{eq:normW}
		\norm{\bfv}{W} = \left(\norm{\bfv}{\Hcurl[\Omega]}^2 + \norm{\gamma_t(\bfv)}{0,\Gamma}^2\right)^{1/2}.
	\end{equation}
	Notice that the space $M$ defined in \eqref{eq:M} keeps unchanged since the boundary condition $\bfv \times \bfn = 0$ on $\Gamma_1$ implies, together with the condition $\bfv \in V$ that $M \subset W$ and $\norm{\bfv}{V} = \norm{\bfv}{W}$ for any field in $M$.

	In order to obtain the associated relaxed formulation, we have several choices. We obviously could just replace the space $V$ by $W$ in the formulation of problem in \eqref{eq:qr-relaxed}. The existence and uniqueness of a solution to the mixed problem on $W$ can be proved in the same way as in the proof of \autoref{thm:convergence-relaxed}. A slight modification occurs in the proof of the convergence of the sequence $(\bfE_\alpha,\bfF_\alpha)$. Indeed, this requires that the Cauchy problem \eqref{eq:Cauchy} admits a solution in the modified vector space $W$ and implies more regularity of the limit field $\bfE$ on $\Gamma_1$. The rest of the proof keeps unchanged.

	From a numerical point of view, it seems however appealing to introduce a new parameter $\nu > 0$ that acts as a regularization parameter on the inaccessible part $\Gamma_1$ of the boundary. An appropriate \textit{"tuning"} of the parameters should allow to improve the numerical results.

	In view of the latter remark, we define the relaxed mixed problem with regularization on both the accessible and inaccessible parts as follows:

	\begin{equation}
		\label{eq:qr-relaxed-reg}
		\left\{
		\begin{array}{rcl@{\hspace{4\tabcolsep}}l}
			\multicolumn{4}{l}{\stext[r]{Find} (\bfE_\beta,\bfF_\beta) \in W \times M \stext[l]{such that}} \\
			\multicolumn{1}{l}{\delta\dotprod{\bfE_\beta}{\bfphi}{V} + \nu\dotprod{\gamma_t(\bfE_\beta)}{\gamma_t(\bfphi)}{0,\Gamma_1}} & & & \multirow{2}{*}{$\forall \bfphi \in W,$} \\
			\qquad + \eta^2\dotprod{\gamma_t(\bfE_\beta)}{\gamma_t(\bfphi)}{0,\Gamma_0} + a(\bfphi,\bfF_\beta) &=& \eta^2\dotprod{\bff}{\gamma_t(\bfphi)}{0,\Gamma_0}, & \\
			a(\bfE_\beta,\bfpsi) - \dotprod{\bfF_\beta}{\bfpsi}{W} &=& \ell(\bfpsi),  &\forall \bfpsi \in M,
		\end{array}
		\right.
	\end{equation}
	where the subscript $\beta = (\delta,\nu,\eta)$ indicates that the solution of \eqref{eq:qr-relaxed} depends on the regularization parameters $\delta > 0$ and $\nu > 0$ as well as on the relaxation parameter $\eta > 0$.
	To our knowledge, this relaxed and regularized mixed version of the QR method has never been proposed in the previous studies of the method.

	Notice that the first two terms in the first equation are well defined according to the definition of the space $W$, but that the weight of the different parts of the norm can be chosen independently.

	\begin{theorem}
		\label{thm:convergence-relaxed-reg}
		Let $(\bff,\bfg) \in L^2(\Gamma_0)^3 \times L^2(\Gamma_0)^3$. For any $\beta = (\delta,\nu,\eta)$ such that $\delta > 0$, $\nu > 0$ and $\eta > 0$, problem~\eqref{eq:qr-relaxed-reg} admits a unique solution $(\bfE_\beta,\bfF_\beta) \in W \times M$.

		If, in addition, $(\bff,\bfg)\in C(\eps, \sigma; \Gamma_0)$ and if the corresponding solution $\bfE$ of the Cauchy problem~\eqref{eq:Cauchy} belongs to the space $W$, then for any fixed $\eta > 0$,
		\begin{equation}
			\label{eq:convergence-relaxed-reg}
			\lim_{(\delta,\nu) \to 0} (\bfE_\beta,\bfF_\beta) = (\bfE, 0)
		\end{equation}
		in $W \times M$ whenever the parameters $\delta$ and $\nu$ satisfy the relation
		\begin{equation}
			\label{condition:nu}
			\nu \coloneqq \nu(\delta) = \delta + o(\delta)
		\end{equation}
		such that the ratio $\delta/\nu$ tends to $1$ as $\delta\to 0$.
		The following estimates hold true
		\begin{align}
			\label{eq:estim-relaxed-reg-1}
			\norm{\bfF_\beta}{W} &\leq \sqrt{\max(\delta,\nu)C(\delta,\nu)}\norm{\bfE}{W} \\
			\label{eq:estim-relaxed-reg-2}
			\norm{\gamma_t(\bfE_\beta) - \bff}{0,\Gamma_0} &\leq \frac{\sqrt{\max(\delta,\nu)C(\delta,\nu)}}{\eta} \norm{\bfE}{W}.
		\end{align}
		where $C(\delta,\nu) \coloneqq \dfrac{\max(\delta,\nu)}{\min(\delta,\nu)}$.
	\end{theorem}

	The proof is similar to the one of \autoref{thm:convergence-relaxed} and we only point out the influence of the parameter $\nu$ on the different steps of the proof.

	\begin{proof}
		As before, problem~\eqref{eq:qr-relaxed-reg} can be written in variational form involving a continuous and coercive sesqui-linear form on $W \times M$. The coercivity constant is now given by $\min(\delta,\nu,1)$.

		If, in addition, the Cauchy problem \eqref{eq:Cauchy} admits a solution $\bfE\in W$, the orthogonality relation reads as follows,
		\begin{equation}
			\label{eq:ortho}
			\delta\dotprod{\bfE_\beta}{\bfE_\beta - \bfE}{V} + \nu\dotprod{\gamma_t(\bfE_\beta)}{\gamma_t(\bfE_\beta - \bfE)}{0,\Gamma_1} + \eta^2\norm{\gamma_t(\bfE_\beta) - \bff}{0,\Gamma_0}^2 + \norm{\bfF_\beta}{W}^2 = 0
		\end{equation}
		by applying the same method used for the derivation of the relation~\eqref{eq:tmp1}.

		Developping the scalar products and applying Cauchy-Schwarz' inequality on the real parts yields the following estimation of $\bfE_\beta$
		\[
			\begin{split}
				\delta\norm{\bfE_\beta}{V}^2 + \nu\norm{\gamma_t(\bfE_\beta)}{0,\Gamma_1}^2
				+ \eta^2\norm{\gamma_t(\bfE_\beta) - \bff}{0,\Gamma_0}^2 + \norm{\bfF_\beta}{W}^2
				\leq \delta\norm{\bfE_\beta}{V} \norm{\bfE}{V} + \nu\norm{\gamma_t(\bfE_\beta)}{0,\Gamma_1} \norm{\gamma_t(\bfE)}{0,\Gamma_1}.
			\end{split}
		\]
		In order to get estimates for the $W$-norm, we notice that the left hand side can be minored by $\min(\delta,\nu) \norm{\bfE_\beta}{W}^2$, whereas the right hand side can be majored by $\max(\delta,\nu) \norm{\bfE_\beta}{W} \norm{\bfE}{W}$. Consequently,
		\begin{equation}
			\label{eq:estim-Ebeta}
			\norm{\bfE_\beta}{W} \leq \frac{\max(\delta,\nu)}{\min(\delta,\nu)} \norm{\bfE}{W}.
		\end{equation}
		Under the given assumptions, the sequence $(\bfE_\beta)_\beta$ is thus
		bounded in $W$ and we obtain as before that $F_\beta$ converges to $0$ in $M$ when $(\delta,\nu) \to 0$ since
		\[
		\norm{\bfF_\beta}{W}^2
			\leq \max(\delta,\nu) \norm{\bfE_\beta}{W} \norm{\bfE}{W}
			\leq \max(\delta,\nu)C(\delta,\nu)\norm{\bfE}{W}^2
		\]
		where $C(\delta,\nu)= \dfrac{\max(\delta,\nu)}{\min(\delta,\nu)}$ is bounded. This yields estimate \eqref{eq:estim-relaxed-reg-1}.
		In the same way, we get \eqref{eq:estim-relaxed-reg-2}.

		As before, we prove that the sequence $(\bfE_\beta)_\beta$ converges weakly to the solution $\bfE$ of the Cauchy problem. In order to show strong convergence, we notice that, according to \eqref{eq:ortho},
		\[\delta\real{\dotprod{\bfE_\beta}{\bfE_\beta - \bfE}{V}} + \nu\real{\dotprod{\gamma_t(\bfE_\beta)}{\gamma_t(\bfE_\beta - \bfE)}{0,\Gamma_1}} \leq 0.\]
		Hence,
		\begin{align*}
			\norm{\bfE_\beta - \bfE}{W}^2 &= \real{\dotprod{\bfE_\beta}{\bfE_\beta - \bfE}{W}} - \real{\dotprod{\bfE}{\bfE_\beta - \bfE}{W}} \\
			&\leq \left(1 - \frac{\delta}{\nu}\right) \real{\dotprod{\bfE_\beta}{\bfE_\beta - \bfE}{V}} - \real{\dotprod{\bfE}{\bfE_\beta - \bfE}{W}} \\
			&\leq \abs*{1 - \frac{\delta}{\nu}} \norm{\bfE_\beta}{V}\norm{\bfE_\beta - \bfE}{V} - \real{\dotprod{\bfE}{\bfE_\beta - \bfE}{W}}
		\end{align*}
		Now, the first term in the last inequality tends to $0$ according to the assumptions on $\delta$ and $\nu$ and since $(\bfE_\beta)_\beta$ is bounded in $W$ and thus in $V$. The second term tends to $0$ since $\bfE_\beta$ converges weakly to $\bfE$ in $W$. This completes the proof.
	\end{proof}

	\begin{remark}
		For the relaxed regularized version of the QR-method and under the regularity assumptions of \autoref{p:discretization}, the discretization error is now of order $\frac{h^s}{\min(\delta,\nu,1)}$ taking into account the coercivity constant of the bilinear form. Again, the boundary term in the $W$-norm is controlled by the $\Hcurl[\Omega]$-norm due to the continuity of the trace operator.
	\end{remark}

	\subsection{Strong formulations and regularity results}

	In this section, we give the strong formulation of the relaxed and regularized QR-method \eqref{eq:qr-relaxed-reg}. This allows us to get a regularity result which yields further insight in the numerical behavior of the discrete solution.
	Taking real-valued test fields in $\mathcal{D}(\Omega)^3$, we see that the solution $(\bfE_\beta,\bfF_\beta)$ satisfies the following partial differential equations on $\Omega$ in the distributional sense:
	\begin{subequations}
		\label{eq:EDP1}
		\begin{eqnarray}
			\delta\curl\curl\bfE_\beta + \delta\bfE_\beta + \curl\curl\overline{\bfF}_\beta - k^2\kappa\overline{\bfF}_\beta &=& 0 \stext{in} \Omega,\\
			\curl\curl\bfE_\beta -k^2\kappa\bfE_\beta - \curl\curl{\bfF}_\beta - {\bfF}_\beta &=& 0 \stext{in} \Omega.
		\end{eqnarray}
	\end{subequations}
	System~\eqref{eq:EDP1} yields
	\begin{subequations}
		\label{eq:EDP2}
		\begin{align}
			\curl\curl(\delta\bfE_\beta + \overline{\bfF}_\beta)&=\bfG \stext{in} \Omega,\\
			\curl\curl(\bfE_\beta - \bfF_\beta) &=\bfH \stext{in} \Omega,
		\end{align}
	\end{subequations}
	with $\bfG = -\delta\bfE_\beta + k^2\kappa\overline{\bfF}_\beta \in L^2(\Omega)^3$ and $\bfH = k^2\kappa\bfE_\beta + \bfF_\beta\in L^2(\Omega)^3$.	For $\delta\neq 1$, it thus follows that $\curl\bfE_\beta$ and $\curl\bfF_\beta$ belong to $\Hcurl[\Omega]$ and the traces $\gamma_t(\curl\bfE_\beta)$, $\gamma_t(\curl\bfF_\beta)$ are well defined as elements in $Y(\Gamma)$.
	Partial integration with appropriated test fields $\Phi$ and $\Psi$ yields the following boundary conditions on $\Gamma_0$ :
	\begin{subequations}
		\label{eq:BC_G0}
		\begin{align}
			\delta\gamma_t(\curl\bfE_\beta) + \delta\gamma_t(\bfE_\beta) + \eta^2\gamma_t(\bfE_\beta) + \gamma_t(\curl\overline{\bfF}_\beta)&=\eta^2\bff \stext{on} \Gamma_0,\\
			\gamma_t(\curl\bfE_\beta) - \gamma_t(\curl\bfF_\beta) - \gamma_t(\bfF_\beta) &= \bfg \stext{on} \Gamma_0.
		\end{align}
	\end{subequations}
	On the inaccessible part $\Gamma_1$ of the boundary, we get in a similar way
	\begin{equation}
		\label{eq:BC_G1}
		\delta\gamma_t(\curl\bfE_\beta) + \nu\gamma_t(\bfE_\beta) +\gamma_t(\curl\overline{\bfF}_\beta) = 0 \stext{on} \Gamma_1.
	\end{equation}
	It remains to study the divergence of the fields $\bfE_\beta$ and $\bfF_\beta$. Taking test fields that are the gradients of some regular scalar potentials in $H^1_0(\Omega)$, we get
	\begin{subequations}
		\label{eq:div}
		\begin{align}
			\delta\div \bfE_\beta - k^2\div(\kappa\overline{\bfF}_\beta) &= 0 \stext{in} \Omega,\\
			k^2\div(\kappa\bfE_\beta) + \div \bfF_\beta &= 0 \stext{in} \Omega.
		\end{align}
	\end{subequations}
	The following regularity result in Lipschitz domains is due to Costabel \cite{Costabel90}.
	\begin{lemma}
		\label{t:costabel}
		Let $\bfu$ be a vector field such that
		$\bfu\in \Hcurl[\Omega]$ and $\div\bfu\in L^2(\Omega)$.
		Then $\bfu\times\bfn\in L^2(\Gamma)^3$ if and only if $\bfu\cdot\bfn\in L^2(\Gamma)$. Further, if one of these conditions is satisfied, $\bfu$ belongs to $H^{1/2}(\Omega)^3$.
	\end{lemma}
	We are now able to state a regularity result for the solution of the RR-QR method in the case of a constant coefficient $\kappa$:
	\begin{theorem}
		\label{t:regularity}
		Let $(\bfE_\beta,\bfF_\beta) \in W\times M$ be the solution of problem \eqref{eq:qr-relaxed-reg} with a constant refractive index $\kappa$ and data $(\bff,\bfg)\in L^2(\Gamma_0)^6$. Assume that the parameter $\delta>0$ is such that  $\delta\neq k^4 |\kappa|^2$.
		Then, $\div \bfE_\beta = \div\bfF_\beta = 0$ on $\Omega$ and
		$(\bfE_\beta,\bfF_\beta)\in H^{1/2}(\Omega)^6$.
	\end{theorem}

	\begin{proof}
		For easier reading, denote by $\bfz_R$ (resp. $\bfz_I$) the real (resp. imaginary) part of a complex number $\bfz$ and let $\bfz = \div \bfE_\beta$ and $\bfw = \div \bfF_\beta$. Then, relations \eqref{eq:div} imply that
		$\begin{bmatrix}
		\bfz_R, \bfz_I, \bfw_R, \bfw_I
		\end{bmatrix}^t$ is solution of the linear system $Ax= 0$ with
		\[
			A = \begin{pmatrix}
				\delta & 0 & -k^2\kappa_R & -k^2\kappa_I\\
				0 & \delta & -k^2\kappa_I &  k^2\kappa_R\\
				k^2\kappa_R & -k^2\kappa_I & 1 & 0 \\
				k^2\kappa_I & k^2\kappa_R & 0 & 1
			\end{pmatrix}.
		\]
		A straightforward computation yields
		\[ {\rm det}(A) = (\delta + k^4 |\kappa |^2)(\delta - k^4 |\kappa|^2). \]
		It follows that $\div\bfE_\beta = \div\bfF_\beta = 0$ in $\Omega$ provided
		that $\delta\neq k^4 |\kappa|^2$.
		Applying \autoref{t:costabel} to the solution  of the RR-QR method shows that has $(\bfE_\beta,\bfF_\beta)$ at least $H^{1/2}$-regularity.
	\end{proof}

	\begin{remark}
		If $\kappa$ is a regular function of $W^{1,\infty}(\Omega)$, we can prove in a similar way that $\div\bfE_\beta$ and $\div\bfF_\beta$ belong to $L^2(\Omega)$. Indeed, developping the divergence of the products in \eqref{eq:div} yields a system with the same matrix $A$ and a right-hand side in $L^2(\Omega)$.
		Concerning the particular case of coefficients $\kappa\in L^\infty(\Omega)$ that are piecewise constant with respect to a partition $\mathcal{P} = \{\Omega_p\}_{p=1:P}$ into regular subdomains $\Omega_p$, we get in a first step
		that $\bfE_\beta$ and $\bfF_\beta$ are divergence free on each part $\Omega_p$ provided $\delta\neq k^4|\kappa_p|^2 $ for all $p$. Then, the
		transmission conditions along the interfaces $\Sigma_p$ read
		\begin{subequations}\label{eq:saut}
			\begin{align}
				\delta[\gamma_n(\bfE_\beta)] - k^2[\gamma_n(\kappa\overline{\bfF}_\beta)] &= 0 \stext{on} \Sigma_p,\\
				k^2[\gamma_n(\kappa\bfE_\beta)] + [\gamma_n(\bfF_\beta)] &= 0 \stext{on} \Sigma_p.
			\end{align}
		\end{subequations}
		For $\delta\neq 1$ we get $[\gamma_n\bfE_\beta] = [\gamma_n(\kappa\bfE_\beta)] = 0$ which implies $\gamma_n(\bfE_\beta) = 0$ unless $[\kappa]\neq 0$. Again, $\div\bfE_\beta$ belongs to $L^2(\Omega)$ and the same result holds for $\bfF_\beta$.
	\end{remark}

	\begin{theorem}
		\label{t:regularity_curl}
		Under the assumptions of \autoref{t:regularity} and if in addition
		$\delta\neq 1$, we have $\curl\bfE_\beta\in \Hcurl[\Omega]$ and $\curl\bfF_\beta\in \Hcurl[\Omega]$.
		Further, for any function $\xi\in \mathcal{C}^\infty(\overline{\Omega})$ such that $\xi\equiv 0$ in a vicinity $\mathcal{V}$ of $\overline{\Gamma}_0\cap\overline{\Gamma}_1$, the fields $\curl(\xi\bfE_\beta)$
		and $\curl(\xi\bfF_\beta)$ belong to $\Hcurl[\Omega]\cap H^{1/2}(\Omega)^3$.
	\end{theorem}

	\begin{proof}
		Since the data $\bff$ and $\bfg$ as well as the trace fields $\gamma_t(\bfE_\beta)$ and $\gamma_t(\bfF_\beta)$ belong to $L^2(\Gamma_0)^3$, we may deduce from \eqref{eq:BC_G0} that $\gamma_t(\curl\bfE_\beta)$ and $\gamma_t(\curl\bfF_\beta)$ belong to $L^2(\Gamma_0)^3$ provided that $\delta\neq 1$.
		On the inaccessible part $\Gamma_1$, we have $\gamma_t(\bfF_\beta) = 0$ on $\Gamma_1$ since
		$\bfF_\beta$ belongs to the vector space $M$. According to a classical differential identity on surfaces, this implies $\gamma_n(\curl\bfF_\beta) = 0$ where $\gamma_n$ is the extension of the normal trace operator.
		Now, let $\eta$ be a smooth function that vanishes near the accessible part $\Gamma_0$ and let $\bfw = \curl\bfF_\beta$. We have $\gamma_n(\eta\bfw) = 0$ on $\Gamma$. A straightforward computation shows that $\curl(\eta\bfw)\in L^2(\Omega)^3$ and $\div(\eta\bfw)\in L^2(\Omega)$. According to \autoref{t:costabel}, the condition $\gamma_n(\eta\bfw)\in L^2(\Gamma)^3$ implies $\gamma_t(\eta\bfw)\in L^2(\Gamma)^3$. Since we already proved that $\gamma_t(\bfw)$ belongs to
		$L^2(\Gamma_0)^3$, we get the same result for $\xi\bfw$ with a function $\xi$ that vanishes only near the intersection
		of $\overline{\Gamma}_0$ and $\overline{\Gamma}_1$. Applying once again Costabel's regularity result yields $\xi\bfw\in H^{1/2}(\Omega)^3$.
		The result for $\curl\bfE_\beta$ can be obtained in a similar way, noticing that $\gamma_t(\curl\bfE_\beta) = -\nu\gamma_t(\bfE_\beta) -\gamma_t(\curl\overline{\bfF}_\beta)$ on $\Gamma_1$ according to \eqref{eq:BC_G1}.
	\end{proof}

	\begin{remark}
		The above regularity results do not yield global $H^{1/2}$-regularity for
		the curl of the fields $\bfE_\beta$ and $\bfF_\beta$. This probably explains the singular behavior of the discrete solutions near the intersection between $\overline{\Gamma}_0$ and $\overline{\Gamma}_1$ observed in the numerical simulations.
	\end{remark}

	\subsection{Numerical results with noisy data}

	We use the same physical parameters as in \autoref{sec:numerical_classic}. In the sequel, we describe the generation of synthetic noisy data. To simplify the notation, we consider here that the input data $\bff$ and $\bfg$ are vectors of degrees of freedom. They are perturbed as follows. First, two vectors $\bfb_\bff$ and $\bfb_\bfg$ are generated, following a standard normal distribution. Then, the perturbed data $(\bff^p,\bfg^p)$ are obtained from
	\[\bff^p = \bff + p\frac{\norm{\bff}{}}{\norm{\bfb_\bff}{}}\bfb_\bff, \quad \bfg^p = \bfg +p\frac{\norm{\bfg}{}}{\norm{\bfb_\bfg}{}}\bfb_\bfg,\]
	where $p > 0$ is the applied level of noise and $\| \cdot \|$ the $\ell^2$-norm.

	Numerical results are presented for the second version of the regularized relaxed quasi-reversibility (RR-QR) method~\eqref{eq:qr-relaxed-reg} with $\nu = \delta$ and $p = \pc{5}$ noise.
	Different choices for the value of the parameter $\eta$ have been tested, and for each choice of $\eta$ the error of the QR-method has been evaluated with respect to $\delta$. It was found that the ratio between $\eta$ and the corresponding optimal $\delta$ was constant. In the sequel, the parameter $\eta$ is fixed automatically according to the following procedure. First, we compute the Riesz representative $\bfG^p$ of $\bfg^p$ in the Hilbert space $M$ by solving
	\[\dotprod{\bfG^p}{\bfpsi}{W} = \dotprod{\bfg^p}{\bfpsi}{0,\Gamma_0}, \quad \forall \bfpsi \in M.\]
	Then, we define
	\[\eta = \frac{\norm{\bfG^p}{{W}}}{\norm{\bff^p}{L^2(\Gamma_0)}}.\]
	As suggested in \cite{BR18}, $\eta$ may then be interpreted as a balance term between the data $\bfG^p$ and $\bff^p$ which are given in different measurement units.

	First, let $\Omega$ be the unit disc. In \autoref{fig:err_delta_relax} we show the evolution of the error in $L^2(\Omega)$-norm with respect to $\delta$, for the two configurations of the boundary. For the specific case of $\delta$ minimizing the error, we show this error over the whole domain in \autoref{fig:err_qr_relax} and we list the different errors in \autoref{tab:err_qr_relax}.

	\begin{figure}
		\centering
		\includegraphics[width=0.33\textwidth]{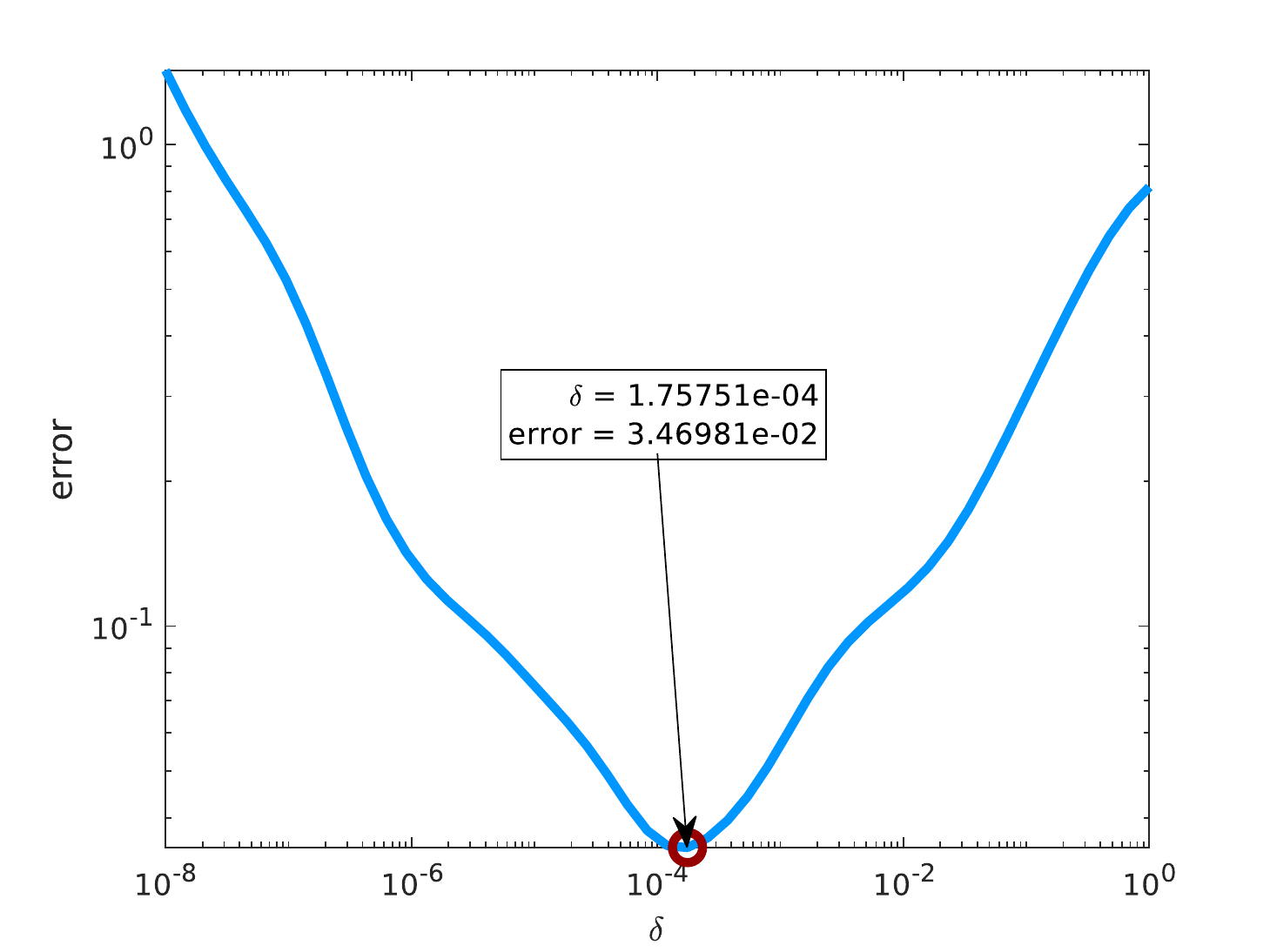}
		\hspace{0.15\textwidth}
		\includegraphics[width=0.33\textwidth]{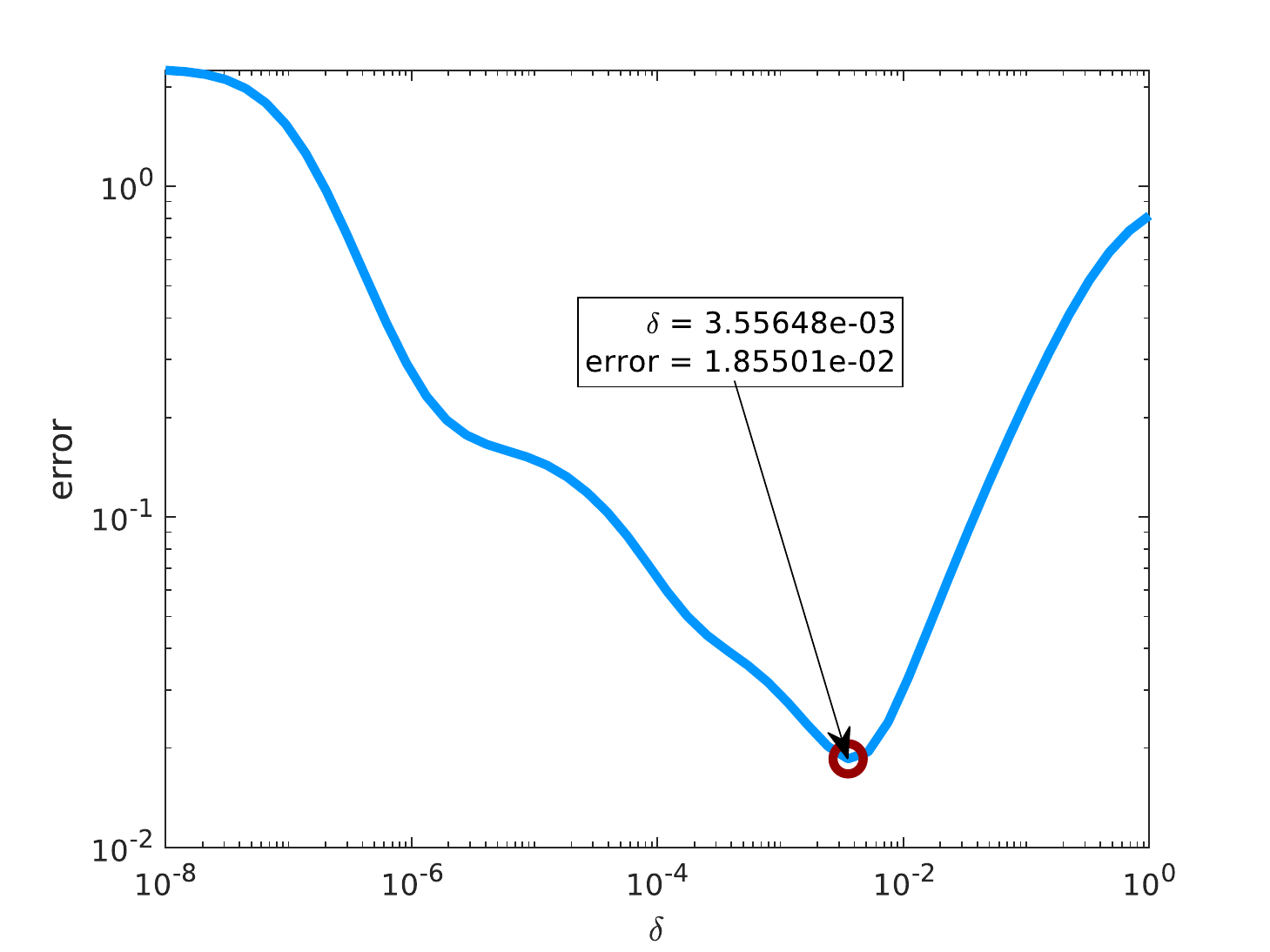}
		\caption{Unit disc. \pc{5} noisy data. RR-QR method. Relative error $\frac{\norm{\bfE - \bfE_\beta}{0,\Omega}}{\norm{\bfE}{0,\Omega}}$ with respect to the regularization parameter $\delta$ at fixed $\eta$ and $\nu=\delta$. Left: configuration G34. Right: configuration GE37.}
		\label{fig:err_delta_relax}
	\end{figure}

	\begin{figure}
		\centering
		\includegraphics[width=0.45\textwidth]{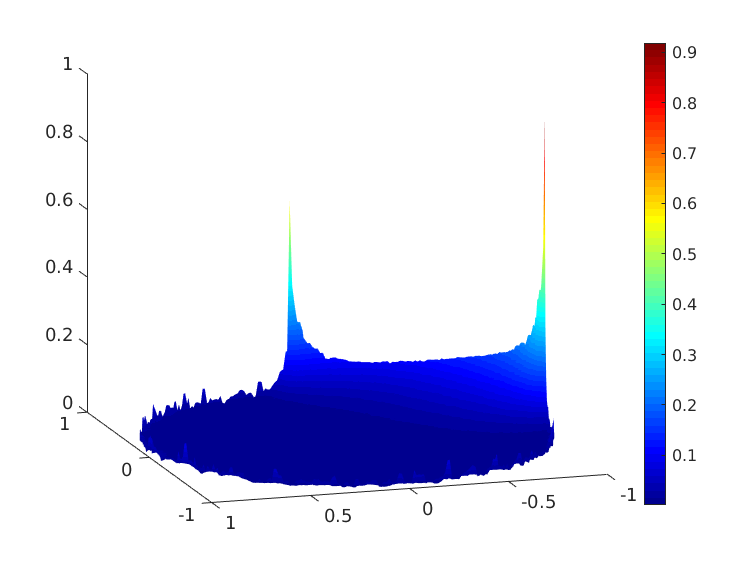}
		\hfill
		\includegraphics[width=0.45\textwidth]{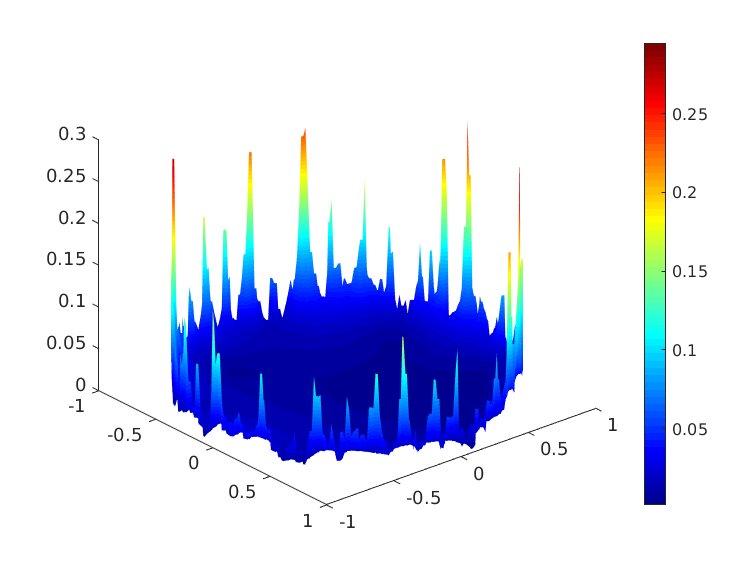}
		\caption{Unit disc. \pc{5} noisy data. RR-QR method. Modulus of the error $\abs{\bfE - \bfE_\beta}$. Left: configuration G34. Right: configuration GE37.}
		\label{fig:err_qr_relax}
	\end{figure}

	\begin{table}
		\centering
		\caption{Unit disc. \pc{5} noisy data. RR-QR method. Errors in the approximation of $\bfE$ for $\nu = \delta$ at optimal delta and automatically fixed $\eta$.}
		\label{tab:err_qr_relax}

		\begin{tabular}{ccccc}
			\toprule
			Configuration & $\frac{\norm{\bfE - \bfE_\delta}{0,\Omega}}{\norm{\bfE}{0,\Omega}}$ & $\frac{\norm{(\bfE - \bfE_\delta) \times \bfn}{0,\Gamma_0}}{\norm{\bfE \times \bfn}{0,\Gamma_0}}$ & $\frac{\norm{(\bfE - \bfE_\delta) \times \bfn}{0,\Gamma_1}}{\norm{\bfE \times \bfn}{0,\Gamma_1}}$ & $\norm{\bfF_\delta}{0,\Omega}$ \\
			\midrule
			G34 & 3.4698e-02 & 2.9145e-02 & 3.5686e-01 & 7.0667e-03 \\
			GE37 & 1.8550e-02 & 3.2848e-02 & 1.7154e-01 & 1.1857e-02 \\
			\bottomrule
		\end{tabular}
	\end{table}

	With the same settings, we now test the regularized relaxed quasi-reversibility method (RR-QR) in the ring with extension/restriction. The relative error in $\bfE$ with respect to $\delta$ is shown in \autoref{fig:err_delta_relax_filled_ring}. For the optimal values of $\delta$, the errors are reported in \autoref{tab:err_qr_relax_filled_ring} and illustrated in \autoref{fig:err_qr_relax_filled_ring}. One notices the good performance of the method for the three configurations with errors below \pc{8} in all norms except for configuration G34 on $\Gamma_1$ where the error amounts to \pc{15}.
	In particular, the error on the interior boundary $\Gamma_i$ is under $5\%$ for the three configurations.
	The method performs better in the ring configuration than in the unit disc.

	\begin{figure}
		\centering
		\includegraphics[width=0.32\textwidth]{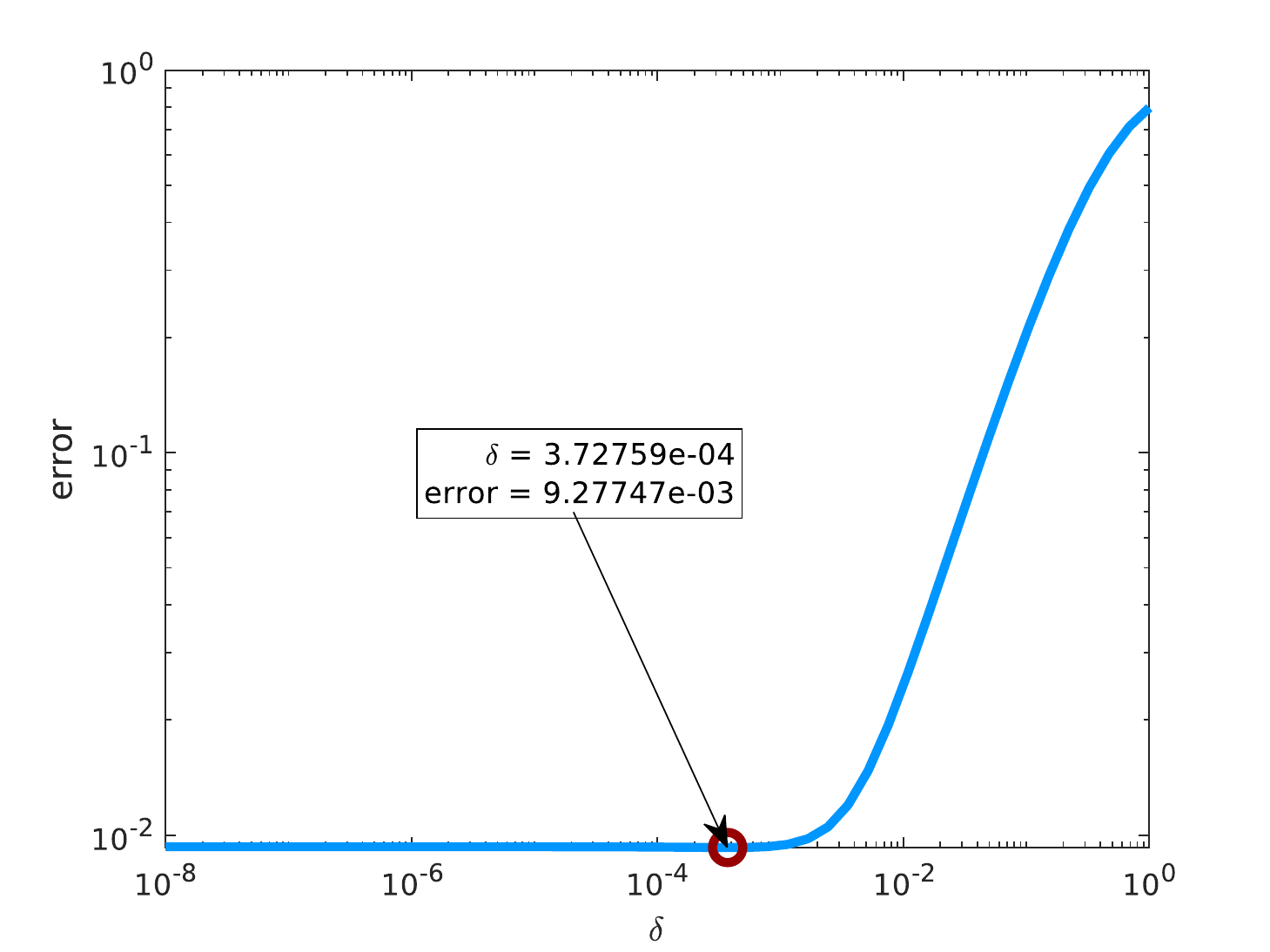}
		\hfill
		\includegraphics[width=0.32\textwidth]{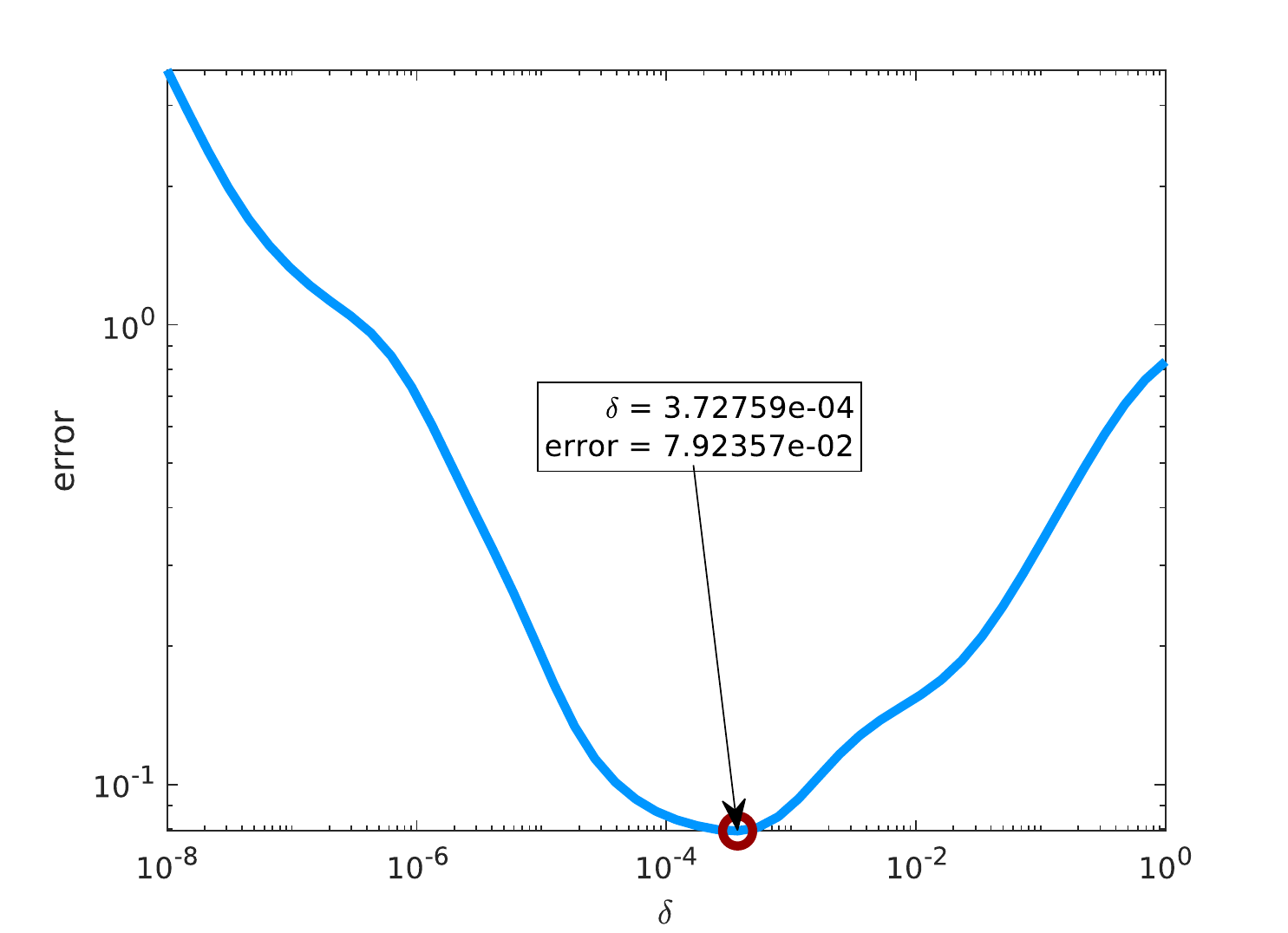}
		\hfill
		\includegraphics[width=0.32\textwidth]{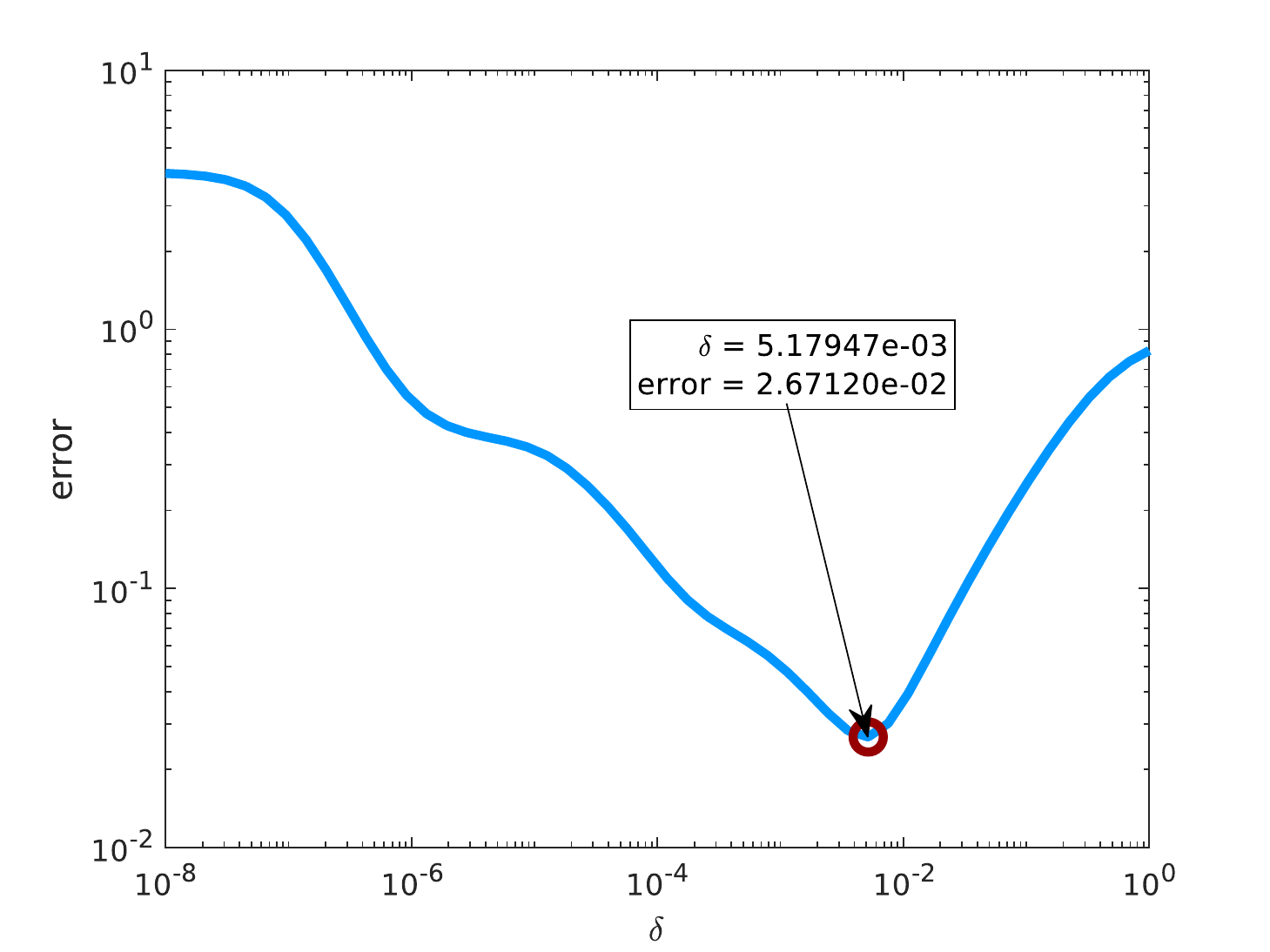}
		\caption{Ring. \pc{5} noisy data. RR-QR method with extension/restriction. Relative error for $\bfE$ in $L^2(\Omega)$-norm with respect to $\delta$ for $\nu = \delta$. $\eta$ automatically fixed from noise level.}
		\label{fig:err_delta_relax_filled_ring}
	\end{figure}

	\begin{figure}
		\centering
		\includegraphics[width=0.32\textwidth]{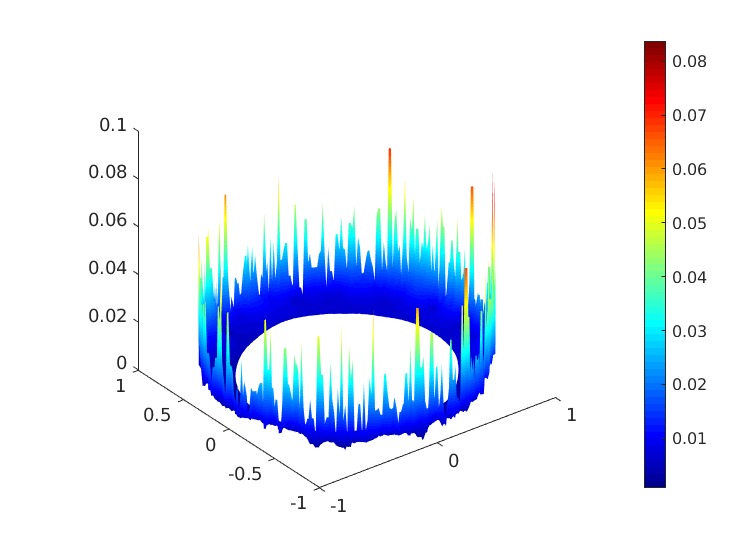}
		\hfill
		\includegraphics[width=0.32\textwidth]{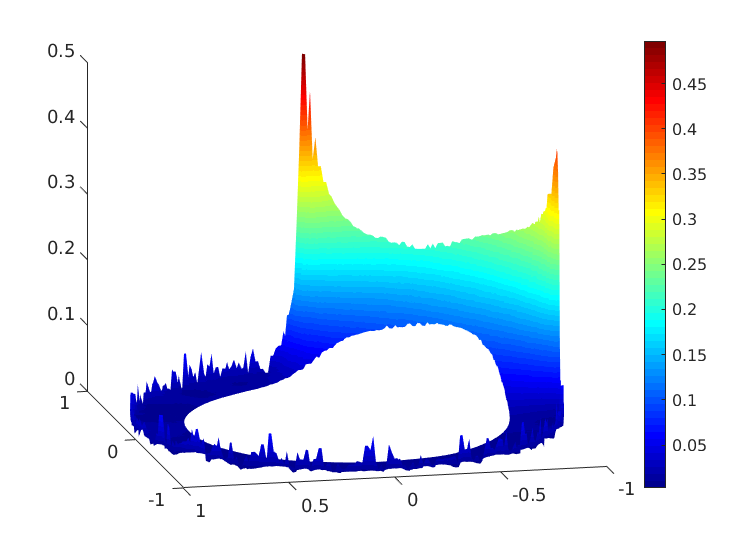}
		\hfill
		\includegraphics[width=0.32\textwidth]{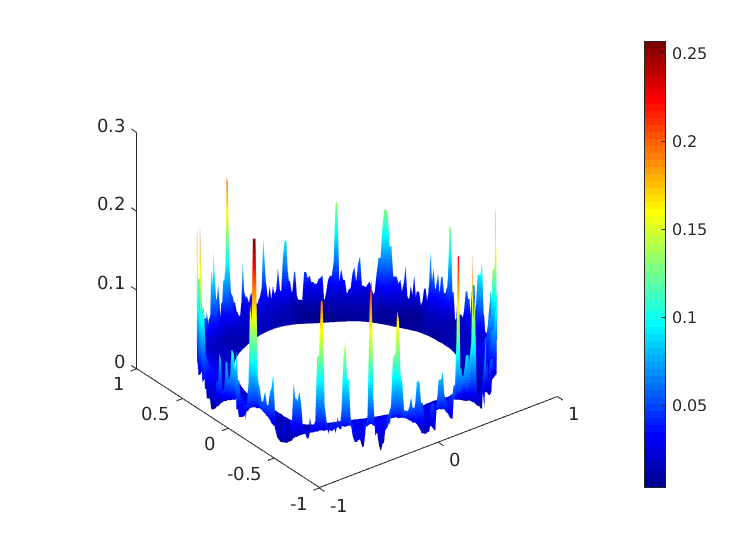}
		\caption{Ring. \pc{5} noisy data. RR-QR method with extension/restriction. Modulus of the error in the approximation of $\bfE$ for optimal $\delta$ and $\nu = \delta$. $\eta$ automatically fixed from noise level. Left: configuration GExt. Middle: configuration G34. Right: configuration GE37.}
		\label{fig:err_qr_relax_filled_ring}
	\end{figure}

	\begin{table}
		\centering
		\caption{Ring. \pc{5} noisy data. RR-QR method with extension/restriction. Errors in the approximation of $\bfE$ for optimal $\delta$ and $\nu=\delta$. $\eta$ automatically fixed from noise level. Errors on parts of the boundary refer to the $L^2$-norm of the tangential trace.}
		\label{tab:err_qr_relax_filled_ring}

		\begin{tabular}{cccccc}
			\toprule
			Configuration & $\frac{\norm{\bfE - \bfE_\delta}{0,\Omega}}{\norm{\bfE}{0,\Omega}}$ & $\frac{\norm{(\bfE - \bfE_\delta) \times \bfn}{0,\Gamma_0}}{\norm{\bfE \times \bfn}{0,\Gamma_0}}$ & $\frac{\norm{(\bfE - \bfE_\delta) \times \bfn}{0,\Gamma_1}}{\norm{\bfE \times \bfn}{0,\Gamma_1}}$ & $\frac{\norm{(\bfE - \bfE_\delta) \times \bfn}{0,\Gamma_\text{i}}}{\norm{\bfE \times \bfn}{0,\Gamma_\text{i}}}$ & $\norm{\bfF_\delta}{0,\Omega}$ \\
			\midrule
			GExt & 9.2775e-03 & 3.3871e-02 & 4.2425e-03 & 4.2425e-03 & 1.0602e-02 \\
			G34 & 7.9236e-02 & 3.0600e-02 & 1.5695e-01 & 4.2697e-02 & 8.7901e-03 \\
			GE37 & 2.6712e-02 & 3.3503e-02 & 7.7043e-02 & 1.6996e-02 & 1.3596e-02 \\
			\bottomrule
		\end{tabular}
	\end{table}

	\subsection{Numerical results in three dimensions}

	We finally present tests in 3D. The physical settings remain the same as in the previous section, and we generate a noise of \pc{5} in the data. The parameter $\eta$ is automatically fixed following the above mentioned procedure, and we choose $\nu = \delta$.

	Here,  the domain $\Omega$ is a subset of $\R^3$. Two domains are tested. The first one is the unit ball of $\R^3$, discretized with a mesh size $h = \scnum{1.43e-1}$, ending up with \num{90077} tetrahedrons and \num{114347} edges. The accessible part, $\Gamma_0$, is defined as the union of 128 electrodes covering approximatively \pc{61} of the whole boundary. This boundary configuration is illustrated in \autoref{fig:GammaE128}.

	\begin{figure}
		\centering
		\includegraphics[width=0.3\textwidth]{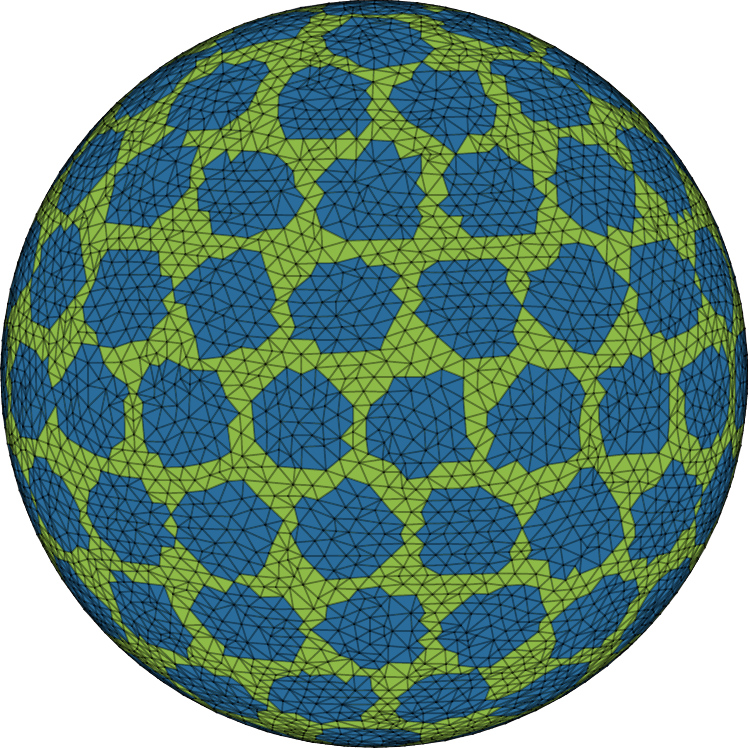}
		\caption{Accessible part $\Gamma_0$ in 3D configurations. Set of 128 electrodes.}
		\label{fig:GammaE128}
	\end{figure}

	The second domain is a ring of width 0.3, defined as the unit ball of $\R^3$, minus the ball centered at the origin and of radius 0.7. The mesh size is close to the ball's one: $h = \scnum{1.41e-1}$, ending up with \num{70420} tetrahedrons and \num{96686} edges. The boundary configuration is the same as for the unit disc: the interior boundary of the ring is part of $\Gamma_1$ and $\Gamma_0$ represents here about \pc{41} of the whole boundary. The quasi-reversibility system will be solved in this 3D ring with the extension/restriction method introduced in \autoref{algo:extension-restriction}.

	We show in \autoref{fig:err_delta_3d} the evolution of the relative error in $L^2$-norm in the whole domain with respect to $\delta$. As in the two dimensional cases, the error decreases with $\delta$, reaches a minimum, and then increases. \autoref{tab:err_3d} lists the errors obtained with the value of $\delta$ realizing this minimum. The error $\norm{\bfE_\delta - \bfE}{0,\Omega}$ is shown in this case in \autoref{fig:err_qr_3d}.
	One notices that despite the larger mesh size compared to the 2D configuration, the approximation is still very good, particularly on the inner boundary in the ring configuration where an error of \pc{5} may be obtained.
	A possible way to consider finer meshes would be to implement a domain decomposition method for 3D Maxwell equations but this is beyond the scope of this paper.

	\begin{figure}
		\centering
		\includegraphics[width=0.35\textwidth]{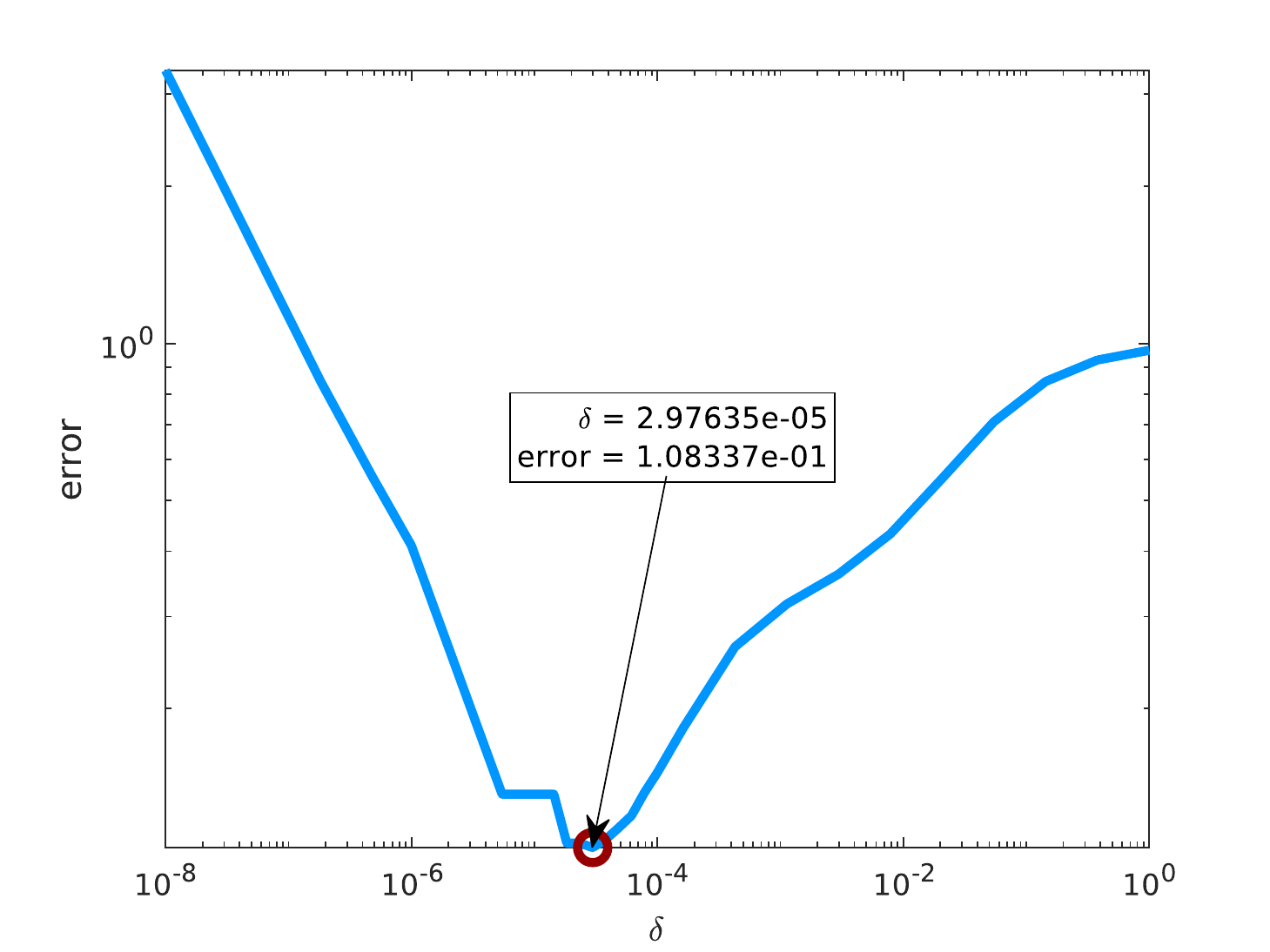}
		\hspace{0.15\textwidth}
		\includegraphics[width=0.35\textwidth]{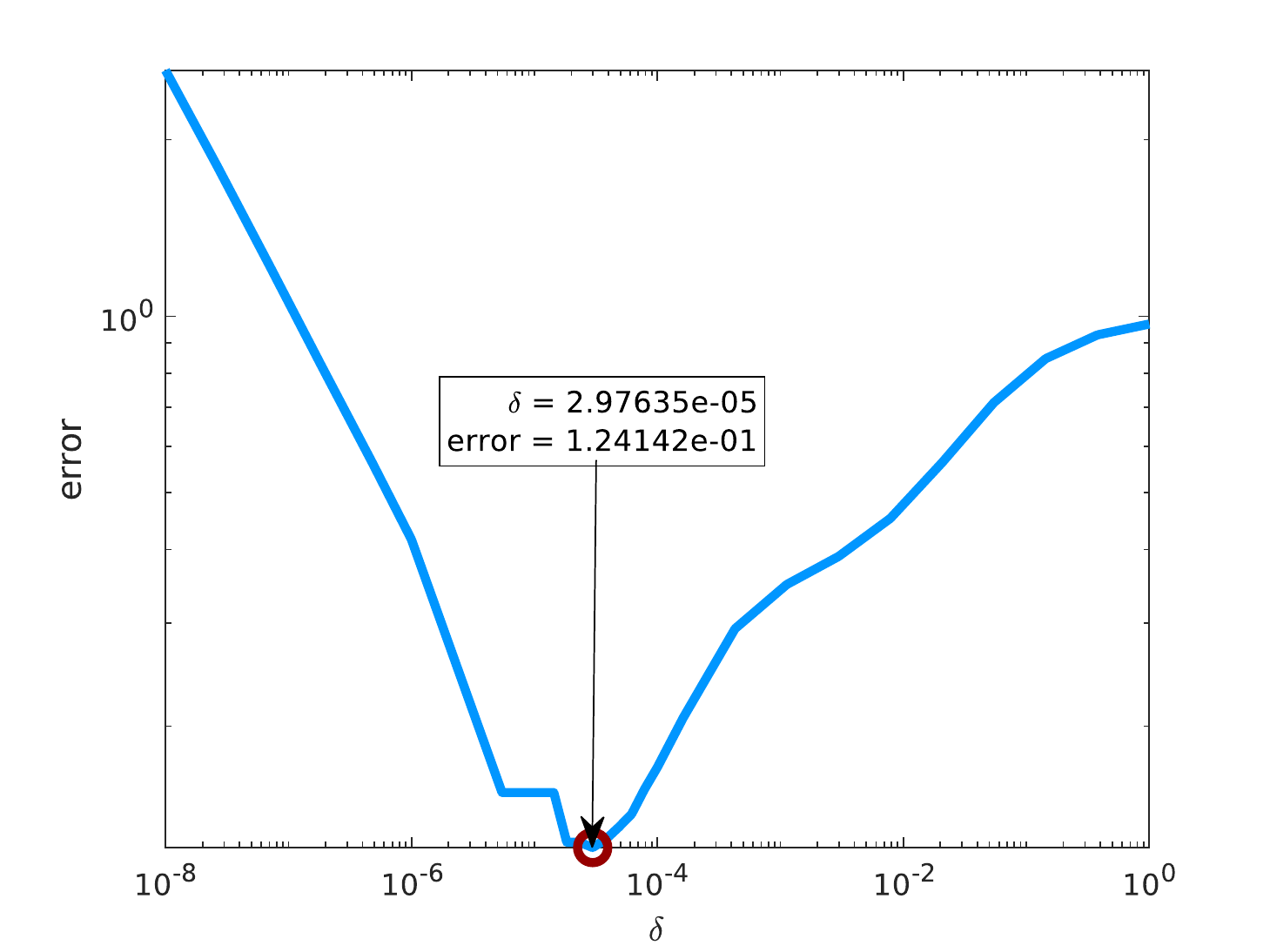}
		\caption{Relative error $\frac{\norm{\bfE - \bfE_\delta}{0,\Omega}}{\norm{\bfE}{0,\Omega}}$ with respect to the regularization parameter $\delta$. Left: Unit ball of $\R^3$. Right: 3D ring of internal radius 0.7.}
		\label{fig:err_delta_3d}
	\end{figure}

	\begin{table}
		\centering
		\caption{Errors in the unit ball of $\R^3$ and in a three dimensional ring. Errors on parts of the boundary refer to the $L^2$-norm of the tangential trace.}
		\label{tab:err_3d}

		\begin{tabular}{cccccc}
			\toprule
			Domain & $\frac{\norm{\bfE - \bfE_\delta}{0,\Omega}}{\norm{\bfE}{0,\Omega}}$ & $\frac{\norm{(\bfE - \bfE_\delta) \times \bfn}{0,\Gamma_0}}{\norm{\bfE \times \bfn}{0,\Gamma_0}}$ & $\frac{\norm{(\bfE - \bfE_\delta) \times \bfn}{0,\Gamma_1}}{\norm{\bfE \times \bfn}{0,\Gamma_1}}$ & $\frac{\norm{(\bfE - \bfE_\delta) \times \bfn}{0,\Gamma_\text{i}}}{\norm{\bfE \times \bfn}{0,\Gamma_\text{i}}}$ & $\norm{\bfF_\delta}{0,\Omega}$ \\
			\midrule
			Ball & 1.0834e-01 & 2.3864e-02 & 4.0949e-01 & $\cdot$ & 2.9336e-03 \\
			Ring & 1.2414e-01 & 1.1567e-02 & 2.7969e-01 & 5.6528e-02 & 1.6497e-03 \\
			\bottomrule
		\end{tabular}
	\end{table}

	\begin{figure}
		\includegraphics[width=0.48\textwidth]{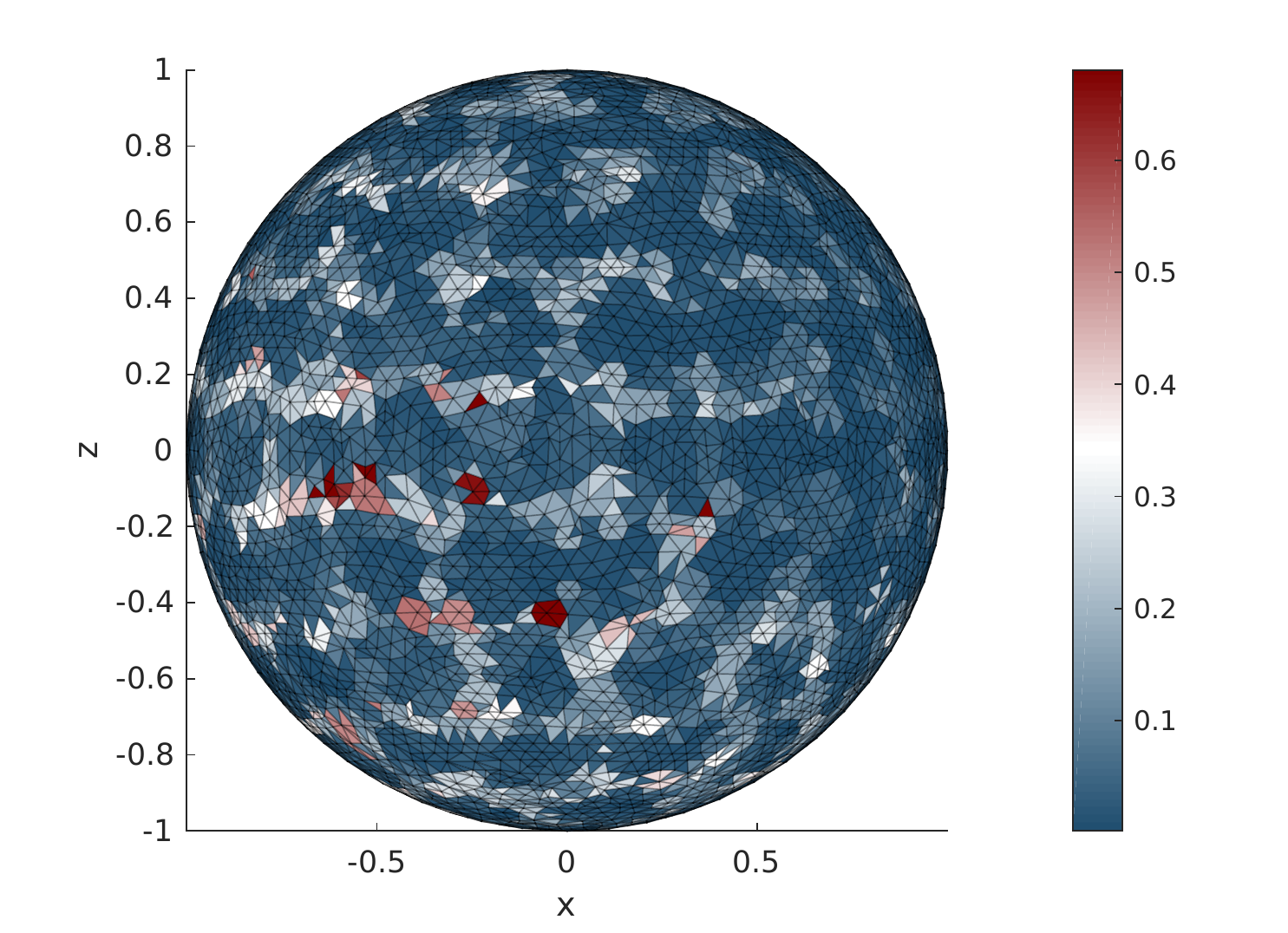}
		\hfill
		\includegraphics[width=0.48\textwidth]{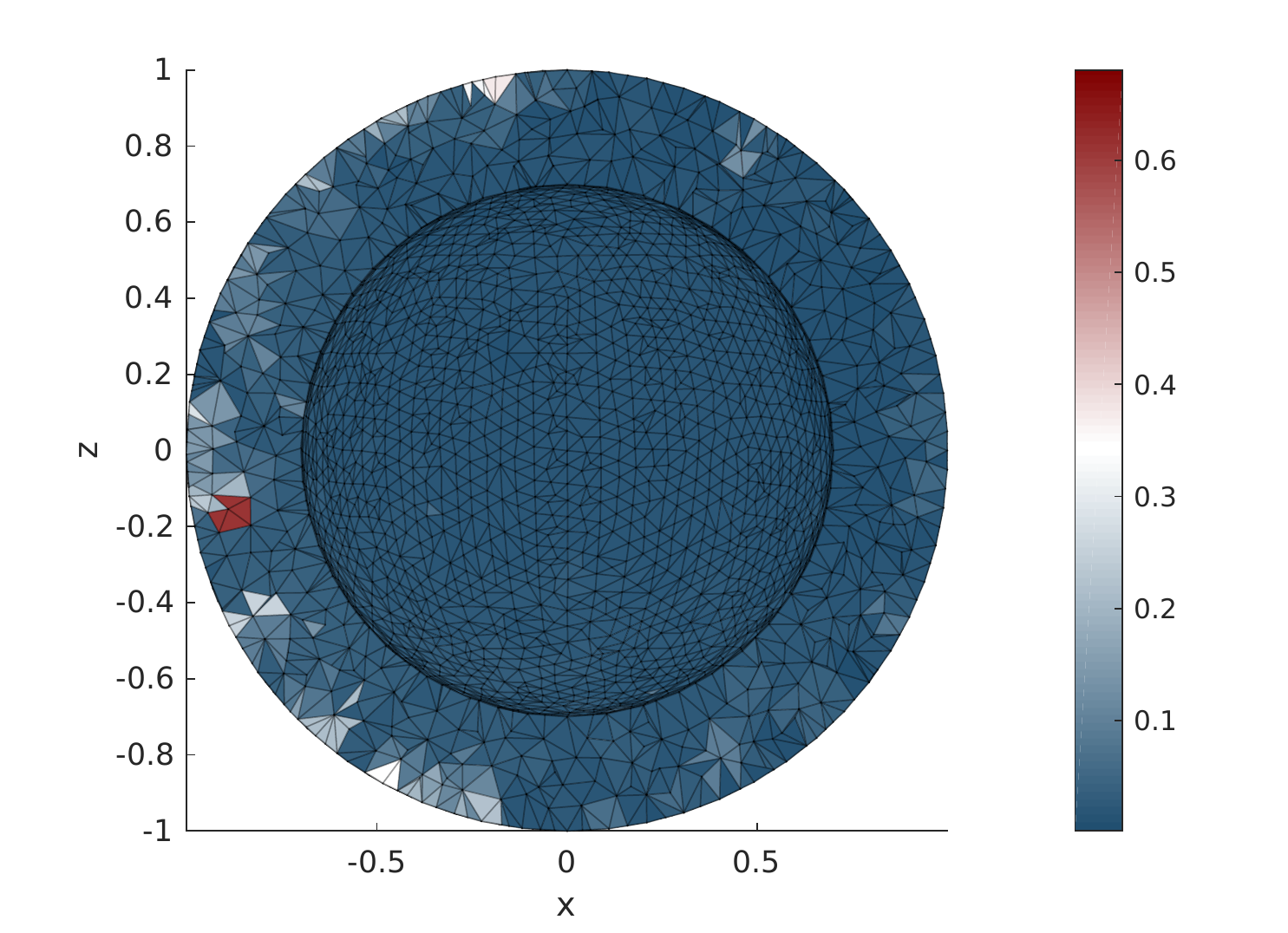}
		\caption{Error in the three-dimensional ring. Left: view of the external boundary. Right: cut showing the internal boundary.}
		\label{fig:err_qr_3d}
	\end{figure}

	\subsection{Choice of the regularization parameter. General case.}

	From the former numerical results, we can attest that the QR-method in its classical or regularized and relaxed version is a reliable tool for data completion and data transmission problems in two and three dimensional configurations.

	The choice of the regularization parameter $\delta$ is crucial for the quality of the approximation. In the previous tests, the parameter has been fixed to an optimal value by evaluating the error in the $L^2(\Omega)$-norm for each value in a given range. In practice, when the exact solution is not known, this procedure is not possible, but classical strategies as the Morozov's discrepancy principle or the L-curve method can be applied to retrieve a value for $\delta$. The L-curve method consists in drawing the graph of $\norm{\bfE_\delta}{0,\Omega}$ with respect to $\norm{\bfF_\delta}{0, \Omega}$ for different values of $\delta$. We then choose $\delta$ to minimize both norms as much as possible. In \autoref{fig:l_curve}, we show the L-curve we obtained in the ring with \pc{5} noise for the configuration G37 and the three-dimensional ring. Interestingly, the minimum of the error on $\Gamma_i$, highlighted by a circle, is located at the corner of the L-curve. This indicates that, in the tested cases, the L-curve method yields a good regularization parameter which can then be automatically computed with, for example, the triangle method~\cite{CastellanosGomezGuerra02}. The convergence of the L-curve method when noise tends to zero is however not addressed. As it is known that this method shows limitations in some cases~\cite{Hanke96}, it could be interesting to compare with other strategies (see for instance~\cite{Colton,Bazan}) but such comparisons are beyond the scope of this paper. We refer to \cite{Vogel}, Chapter 7, where different methods are compared in the context of a linear least square problem.

	\begin{figure}[hbt]
		\centering
		\includegraphics[width=0.4\textwidth]{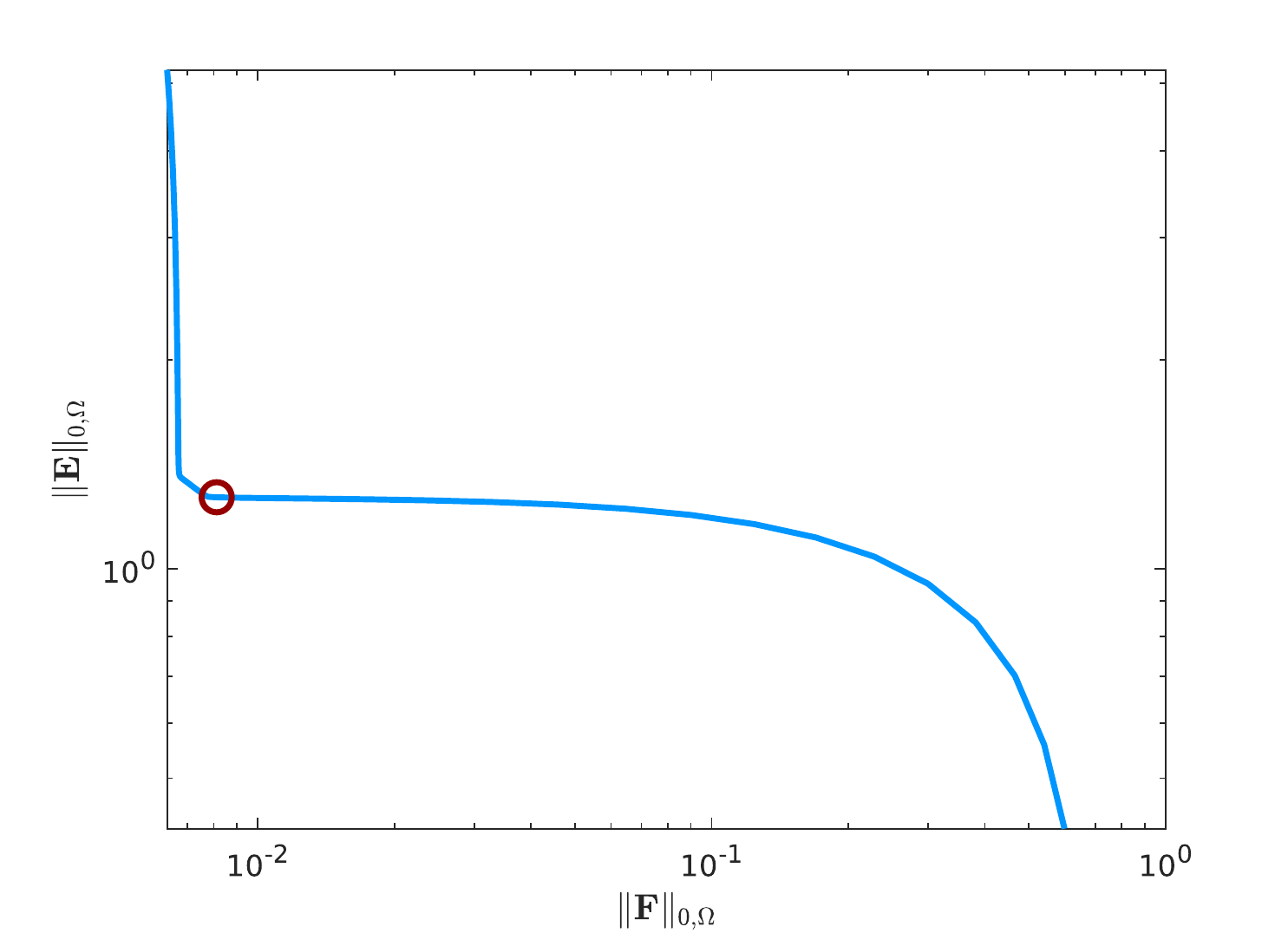}
		\includegraphics[width=0.4\textwidth]{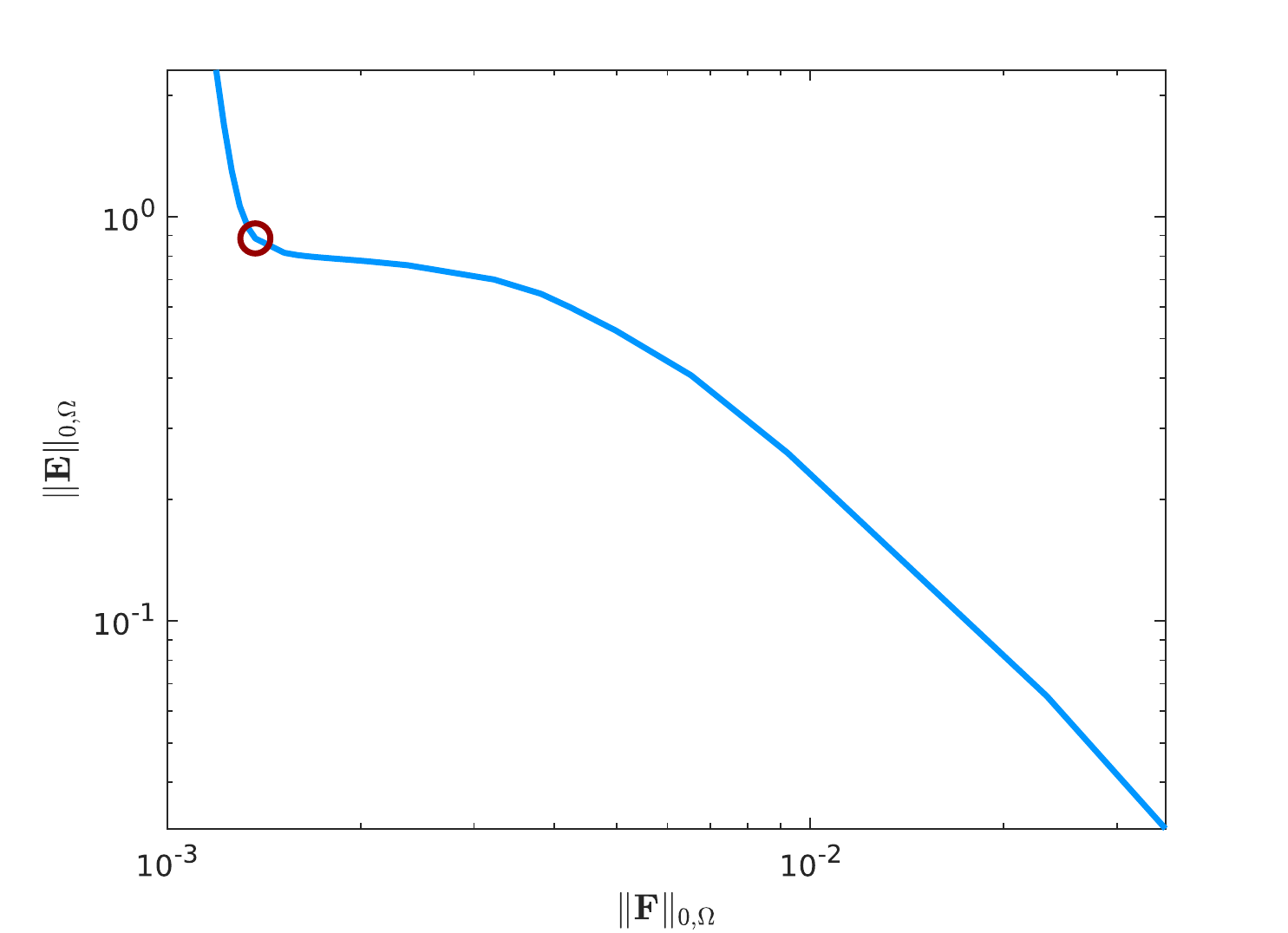}
		\caption{L-curve corresponding to the RR-QR method with extension/restriction in the ring. \pc{5} noisy data. Left: configuration G37 (2D). Right: 3D ring.}
		\label{fig:l_curve}
	\end{figure}

	The choice of the second regularization parameter $\nu$ should been made in accordance with \autoref{thm:convergence-relaxed-reg}, i.e. $\nu \coloneqq \nu(\delta)=\delta + o(\delta)$. However, this condition has only been proven to be sufficient to assure the convergence of the RR-QR method, and other choices of $\nu$ could possibly improve the results. Finally, in all our numerical experiments with the RR-QR method, the relaxation parameter $\eta$ has been fixed according to the heuristical arguments presented here above. Again, this choice is possibly not optimal and could be improved.

	\section{A few words on the link with Tikhonov regularization}
	\label{sec:tikhonov}

	Before concluding this paper, we make a remark on the link between the quasi-reversibility method formulated as a mixed problem and standard Tikhonov regularization, as mentioned in \cite{BR18}.
	We shall make precise the corresponding Tikhonov functionals for the different versions of the QR-method. To this end, denote by $A\colon \Hcurl[\Omega] \to M$ the unique linear continuous operator defined by
	\[a(\bfu,\bfpsi) = \dotprod{A\bfu}{\bfpsi}{\Hcurl[\Omega]}\ \forall \bfpsi \in M\]
	according to the Riesz representation theorem. In the same way, define by $\bfG \in M$ the unique Riesz representative of the continuous linear form $\ell(\cdot)$ such that
	\[\dotprod{\bfG}{\bfpsi}{\Hcurl[\Omega]} = \ell(\bfpsi)\ \forall \bfpsi \in M.\]
	Then, the Cauchy problem~\eqref{eq:Cauchy} consists in finding $\bfE \in \Hcurl[\Omega]$ such that $\gamma_t(\bfE) = \bff$ and $A\bfE = \bfG$.

	Now, for a given parameter $\delta > 0$, introduce the (real-valued) cost function
	\begin{equation}
		\label{def:Jdelta}
		J_\delta(\bfv) = \frac{1}{2} \norm{A\bfv - \bfG}{\Hcurl[\Omega]}^2 + \frac{\delta}{2} \norm{\bfv}{\Hcurl[\Omega]}^2
	\end{equation}
	defined on $\Hcurl[\Omega]$. The directional derivative of $J_\delta$ in the direction $\bfd \in V_0$ is given by
	\begin{equation}
		\label{eq:Jprime}
		J_\delta^\prime(\bfv)\bfd = \real{\dotprod{A\bfv - \bfG}{A\bfd}{\Hcurl[\Omega]} + \delta\dotprod{\bfv}{\bfd}{\Hcurl[\Omega]}}.
	\end{equation}
	Now, let $(\bfE_\delta,\bfF_\delta) \in V_\bff \times M$ be the solution of the classical QR-problem~\eqref{eq:qr}. Then, we get from the second equation of \eqref{eq:qr} that $\bfF_\delta = A\bfE_\delta - \bfG$. Substituting $\bfF_\delta$ by this relation in the first equation and taking the real part, yields $J_\delta^\prime(\bfE_\delta) = 0$ according to \eqref{eq:Jprime} and the definition of $A$ and $\bfG$. In conclusion, the unique solution $\bfE_\delta$ of the QR-method is the critical point of the functional $J_\delta$ defined by \eqref{def:Jdelta}.

	In the same way, one can show that the solution $\bfE_\alpha$ of the relaxed QR method~\eqref{eq:qr-relaxed} is the critical point of the functional
	\[J_\alpha(\bfv) = \frac{1}{2}\norm{A\bfv - \bfG}{V}^2 + \frac{\eta^2}{2}\norm{\gamma_t(\bfv) - \bff}{0,\Gamma_0}^2 + \frac{\delta}{2}\norm{\bfv}{V}^2\]
	defined on the vector space $V$ for a parameter set $\alpha = (\delta,\eta)$.

	Finally, if we let
	\[J_\beta(\bfv) = J_\alpha(\bfv) + \frac{\nu}{2}\norm{\gamma_t(\bfv)}{0,\Gamma_1}^2,\]
	for any $\bfv \in W$ and $\beta = (\delta,\eta,\nu)$, the solution $\bfE_\beta$ of the relaxed and regularized QR method~\eqref{eq:qr-relaxed-reg} is the critical point of $J_\beta$.

	\section{Conclusion and future works}

	In this paper, we have presented a first conclusive study of the applicability of the quasi-reversibility method to the numerical resolution of a data completion problem for Maxwell's equations.
	The main issue was to retrieve the missing data on the inaccessible part in a stable way.
	We have addressed the theoretical and numerical analysis of the method in vector spaces that are naturally involved in variational theory for Maxwell's equations. Different mixed formulations have been proposed: a classical one, and two regularized relaxed formulations which are essential to consider noisy data. We have proved their well-posedness and the convergence
	of the regularized solution to the exact solution under some conditions on the regularization and relaxation parameters. We further showed that the solution of the RR-QR method belongs to $H^{1/2}(\curl)$ away from the intersection of accessible and inaccessible parts. Further investigations in the vein of \cite{BourgeoisChesnel} on the involved singularities could help to improve the results and are part of ongoing work. A large variety of two-dimensional and three-dimensional numerical results attest the efficiency
	of the quasi-reversibility method for data completion and data transmission problems in electromagnetics.
	For the case of noisy data, we have studied in detail the numerical choice of the regularization parameters.

	This non-iterative data completion procedure would be an interesting initial step for numerical methods to reconstruct the dielectric properties of a tissue or a material. In \cite{DHL_19}, we developped for instance an algorithm that is able to localize a perturbation of small amplitude in the parameters $\eps$ and $\sigma$ from given total data of the perturbed field on the boundary. In the case of partial data, the results of the present paper could complete the missing information.
	We have in mind brain medical applications (e.g. microwave tomography \cite{Tournier,Tournier2}) and the cortical mapping problem which consists in reconstructing from electric potential measurements available at part of the scalp the potential on the inner cortex surface. This question of multi-layer data completion has been studied for Laplace's equation in the context of the EEG (electroencephalography) inverse source localization~\cite{Clerc12,Clerc16}. To the best of our knowledge, this problem hasn't been treated for electromagnetic inverse medium problems and induces challenging questions from both the theoretical and numerical point of view. This is part of ongoing work.

\end{document}